\documentclass[mathpazo]{jms}

\usepackage{graphicx}
\usepackage{amsmath}
\usepackage{epstopdf} % needed if you have eps figures in a pdflatex manuscript
\usepackage{mathptmx}
\usepackage{stmaryrd}
\usepackage{amssymb}
\usepackage{empheq}
\usepackage{cases}
\usepackage{cite}
\usepackage{caption,graphicx}
\input psfig.sty

\newcommand{\eps}{\varepsilon}
\newcommand{\bN}{\mathbb{N}}
\newcommand{\bR}{\mathbb{R}}
\newcommand{\bC}{\mathbb{C}}
\newcommand{\mcE}{\mathcal{E}}
\newcommand{\fe}{\mathrm{e}}
\newcommand{\nd}{\lambda}
\newcommand{\be}{\begin{equation}}
\newcommand{\ee}{\end{equation}}
\newcommand{\ba}{\begin{array}}
\newcommand{\ea}{\end{array}}
\newcommand{\bea}{\begin{eqnarray}}
\newcommand{\eea}{\end{eqnarray}}
\newcommand{\beas}{\begin{eqnarray*}}
\newcommand{\eeas}{\end{eqnarray*}}
\numberwithin{equation}{section}
\begin{document}
%%%%% title : short title may not be used but TITLE is required.
% \title{TITLE}
% \title[short title]{TITLE}
\title[Multiscale time integrators for oscillatory equations]
{Uniformly accurate multiscale time integrators for highly
oscillatory second order differential equations}

%%%%% author(s) :
% single author:
% \author[name in running head]{AUTHOR\corrauth}
% [name in running head] is NOT OPTIONAL, it is a MUST.
% Use \corrauth to indicate the corresponding author.
% Use \email to provide email address of author.
% \footnote and \thanks are not used in the heading section.
% Another acknowlegments/support of grants, state in Acknowledgments section
% \section*{Acknowledgments}
%\author[O.~Author]{Only Author\corrauth}
%\address{School of Mathematical Sciences, Beijing Normal University,
%Beijing 100875, P.R. China}
%\email{{\tt author@email} (O.~Author)}

% multiple authors:
% Note the use of \affil and \affilnum to link names and addresses.
% The author for correspondence is marked by \corrauth.
% use \emails to provide email addresses of authors
% e.g. below example has 3 authors, first author is also the corresponding
%      author, author 1 and 3 having the same address.
 \author[Weizhu Bao et.~al.]{Weizhu Bao\affil{1}\comma\corrauth,
       Xuanchun Dong\affil{2}, and Xiaofei Zhao\affil{3}}
 \address{\affilnum{1}\ Department of Mathematics and Center for Computational
              Science and Engineering,
          National University of Singapore,
          Singapore 119076, Singapore. \\
           \affilnum{2}\ Beijing Computational Science Research Center, Beijing 100084, P. R. China.\\
           \affilnum{3}\ Department of Mathematics,
          National University of Singapore,
          Singapore 119076, Singapore. \\}
 \emails{{\tt matbaowz@nus.edu.sg} (W.~Bao), {\tt dong.xuanchun@gmail.com} (X.~Dong),
          {\tt zhxfnus@gmail.com} (X.~Zhao)}
% \footnote and \thanks are not used in the heading section.
% Another acknowlegments/support of grants, state in Acknowledgments section
% \section*{Acknowledgments}

%%%%% Begin Abstract %%%%%%%%%%%
\begin{abstract}
In this paper, two multiscale time integrators (MTIs), motivated from two types of multiscale
decomposition by either frequency or frequency and amplitude, are proposed and analyzed for
solving highly oscillatory second order differential equations
with a dimensionless parameter $0<\eps\le1$. In fact, the solution to this equation
propagates waves with wavelength at $O(\eps^2)$ when  $0<\eps\ll 1$, which brings significantly
numerical burdens in practical computation. We
rigorously establish  two independent error bounds for the two MTIs
at $O(\tau^2/\eps^2)$ and $O(\eps^2)$ for $\eps\in(0,1]$ with $\tau>0$ as step size, which
imply that the two MTIs converge uniformly with linear convergence rate
at $O(\tau)$ for  $\eps\in(0,1]$ and optimally with quadratic convergence rate at $O(\tau^2)$
in the regimes when either $\eps=O(1)$ or $0<\eps\le \tau$.
Thus the meshing strategy requirement (or $\eps$-scalability) of the two MTIs
is $\tau=O(1)$ for $0<\eps\ll 1$,
which is significantly improved from $\tau=O(\eps^3)$ and $\tau=O(\eps^2)$
requested by finite difference methods and exponential wave integrators to
the equation, respectively.  Extensive numerical tests and comparisons with
those classical numerical integrators are reported, which gear towards better
understanding  on the convergence and resolution properties of  the two MTIs.
In addition, numerical results support the two error bounds very well.
\end{abstract}
%%%%% end %%%%%%%%%%%

%%%%% AMS/Chinese Library Classifications/Keywords %%%%%%%%%%%
 \clc{O2} 

%%%% maketitle %%%%%
\maketitle

%%%% Start %%%%%%
\section{Introduction}
\label{sec1}
This paper is devoted to the study of numerical solutions of the following highly
oscillatory second order  differential equations (ODEs)
\begin{equation}\label{WODEs}
\left\{
  \begin{split}
    & \eps^2\ddot{\mathbf{y}}(t)+A\mathbf{y}(t)+\frac{1}{\eps^2}\mathbf{y}(t)
    +\mathbf{f}\left(\mathbf{y}(t)\right)=0,\quad t>0,\\
    & \mathbf{y}(0)=\Phi_1,\quad\dot{\mathbf{y}}(0)=\frac{\Phi_2}{\eps^2}.
  \end{split}
\right.
\end{equation}
Here $t$ is time, $\mathbf{y}:=\mathbf{y}(t)=(y_1(t),\ldots,y_d(t))^T\in\bC^d$ is a
complex-valued vector function with $d$ a positive integer, $\dot{\mathbf{y}}$ and $\ddot{\mathbf{y}}$
refer to the first and second order derivatives of $\mathbf{y}$, respectively, $0<\eps\leq 1$
is a dimensionless parameter which can be very small in some limit regimes,
$A\in\bR^{d\times d}$ is a symmetric nonnegative definite matrix,
$\Phi_{1},\ \Phi_{2} \in\bC^d$ are two given
initial data at $O(1)$ in term of  $0<\eps\ll 1$, and
$\mathbf{f}(\mathbf{y})=(f_1(\mathbf{y}),\ldots,f_d(\mathbf{y}))^T:\ \bC^d \to\bC^d$
describes the nonlinear interaction and it is
independent of $\eps$. The {\sl gauge invariance} implies that $\mathbf{f}(\mathbf{y})$
satisfies the following  relation \cite{Masmoudi}
\begin{equation}\label{gauge}
\mathbf{f}(\fe^{is}\mathbf{y})=\fe^{is}\mathbf{f}(\mathbf{y}), \qquad \forall s\in \bR.
\end{equation}
We remak that when the initial data $\Phi_1,\Phi_2\in\bR^d$ and $\mathbf{f}(\mathbf{y}):\ \bR^d \to\bR^d$,
then the solution $\mathbf{y}\in \bR^d$ is real-valued. In this case,
the gauge invariance condition (\ref{gauge}) for the nonlinearity in
(\ref{WODEs}) is no longer needed.

The above problem is motivated from our recent numerical study of
the nonlinear Klein--Gordon equation in the nonrelativistic limit regime
\cite{Dong,Machihara, Masmoudi}, where $0<\eps\ll 1$ is scaled to be
inversely proportional to the speed of light. In fact, it can be
viewed as a model resulting from a semi-discretization
in space, e.g., by finite difference or spectral discretization with a fixed mesh size
(see detailed equations (3.3) and (3.19) in \cite{Dong}), to the
nonlinear Klein--Gordon equation.
In order to propose new multiscale time integrators (MTIs) and compare
with those classical numerical integrators including finite difference
methods \cite{Dong, Duncan,Strauss,Reich} and exponential wave integrators
\cite{Garcia,Sanz,Grimm,Lubich1,Lubich2} efficiently, we thus focus on the above second order differential equations instead of the original nonlinear Klein-Gordon equation. The solution to (\ref{WODEs}) propagates
high oscillatory waves with  wavelength at $O(\eps^2)$ and amplitude at $O(1)$.
To illustrate this, Figure \ref{fig:00} shows the solutions of (\ref{WODEs}) with
$d=2$, $f_1(y_1,y_2)=y_1^2y_2$, $f_2(y_1,y_2)=y_2^2y_1$, $A={\rm diag}(2,2)$,
$\Phi_1=(1,0.5)^T$ and $\Phi_2=(1,2)^T$ for different $\eps$.
The highly oscillatory nature of solutions to (\ref{WODEs}) causes
severe burdens in practical computation, making the numerical
approximation extremely challenging and costly in the regime of
$0<\eps\ll 1$.

\begin{figure}
\centerline{\psfig{figure=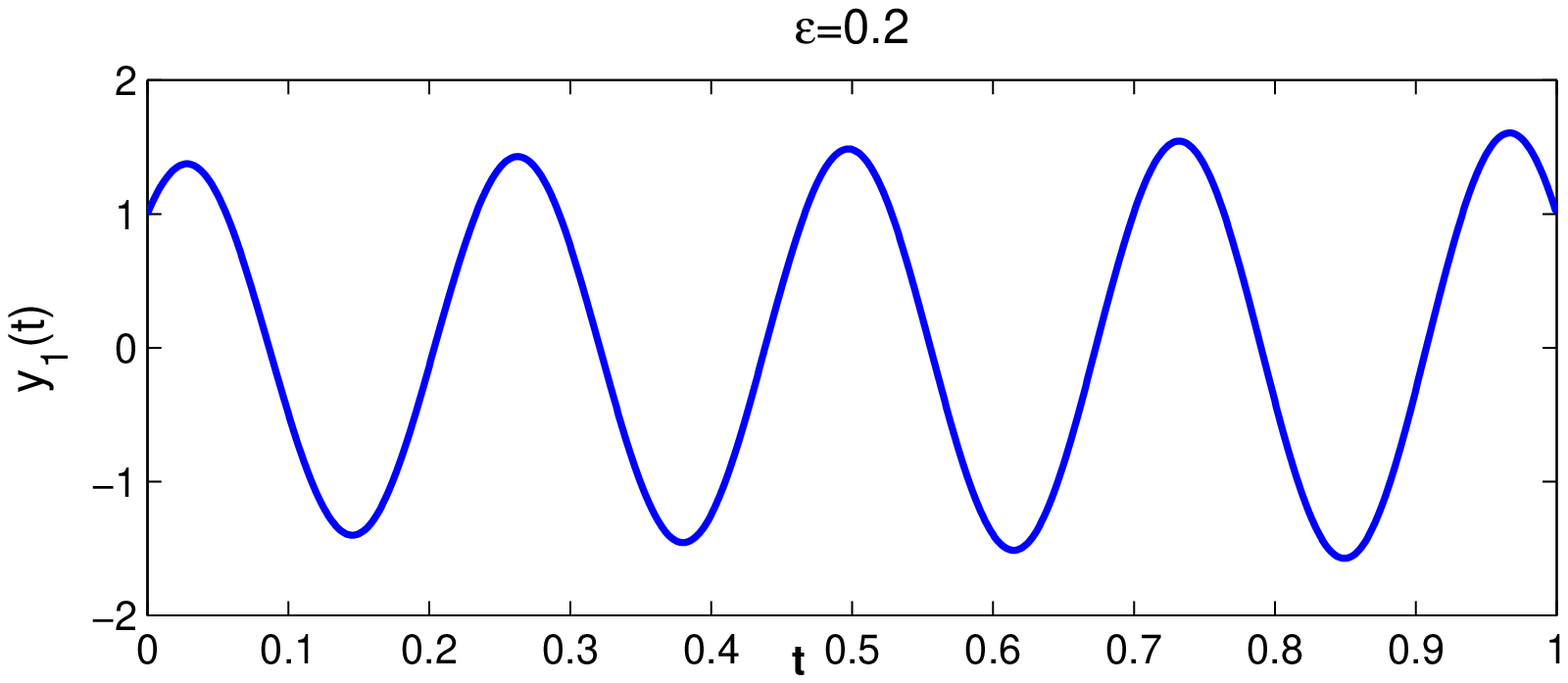,height=4cm,width=7.5cm,angle=0}
\psfig{figure=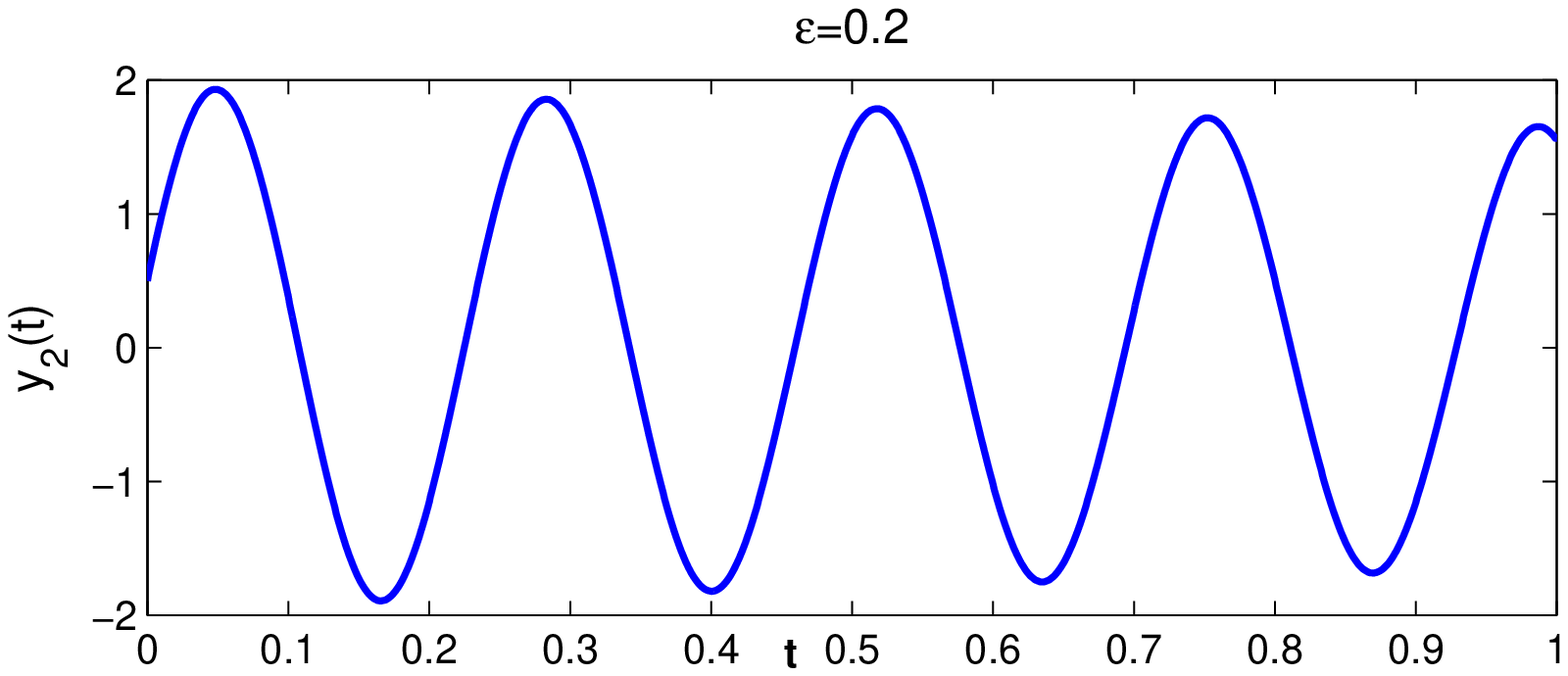,height=4cm,width=7.5cm,angle=0}}
\centerline{\psfig{figure=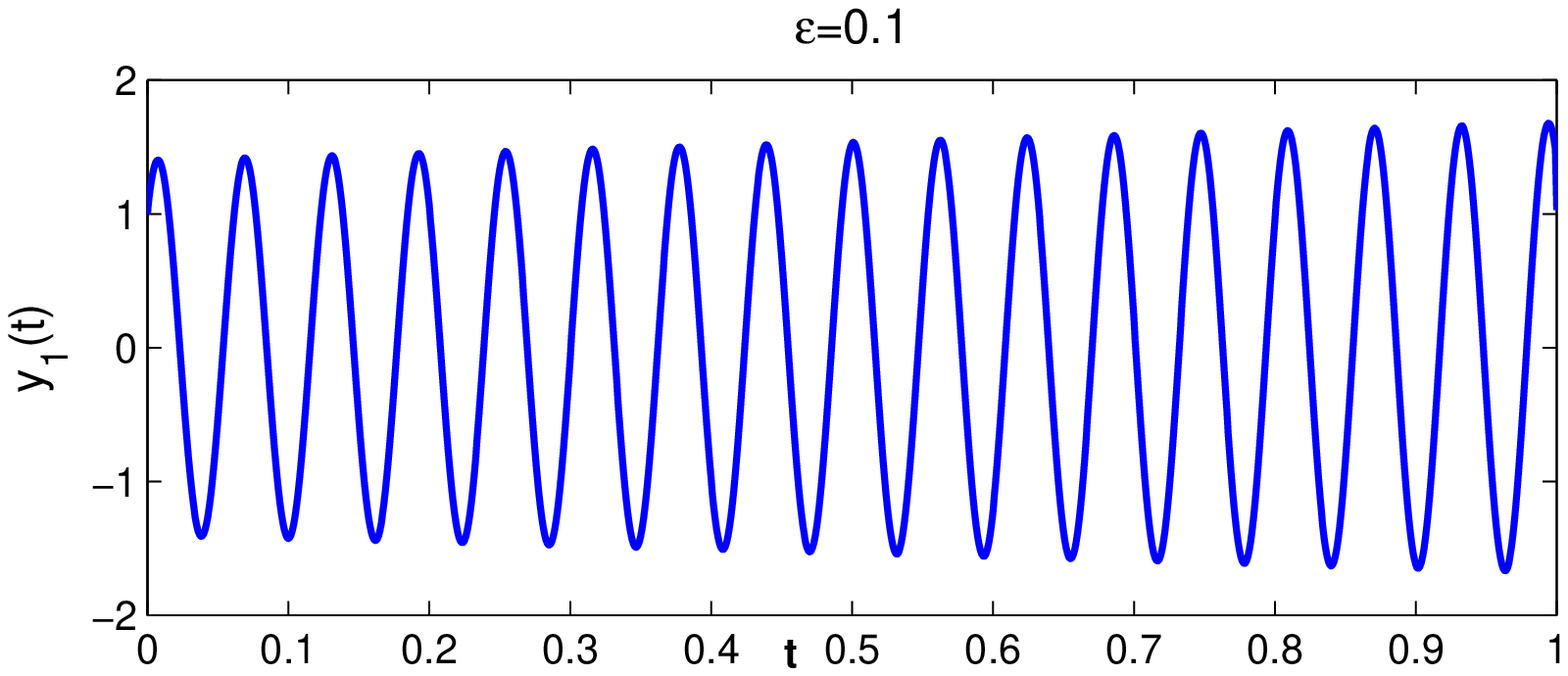,height=4cm,width=17cm}}
\centerline{\psfig{figure=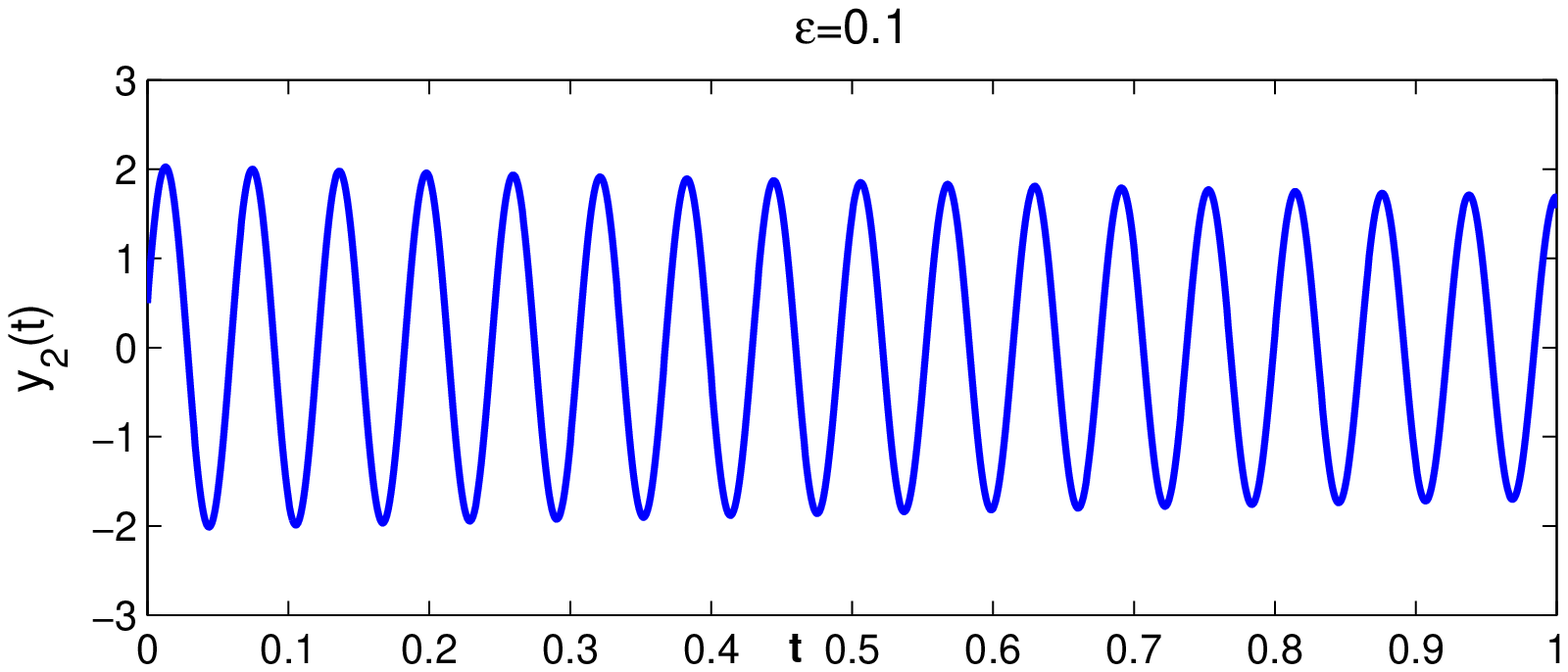,height=4cm,width=17cm}}
\centerline{\psfig{figure=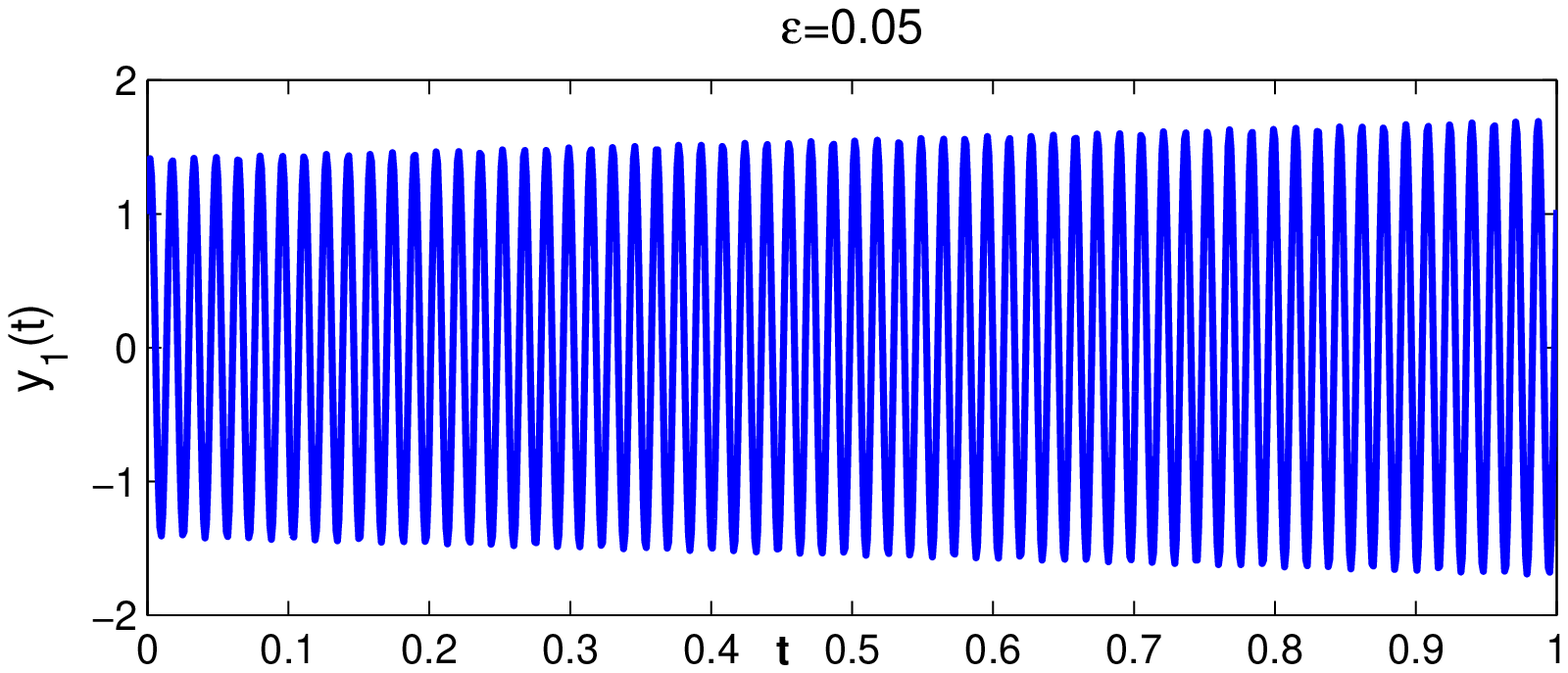,height=4cm,width=17cm}}
\centerline{\psfig{figure=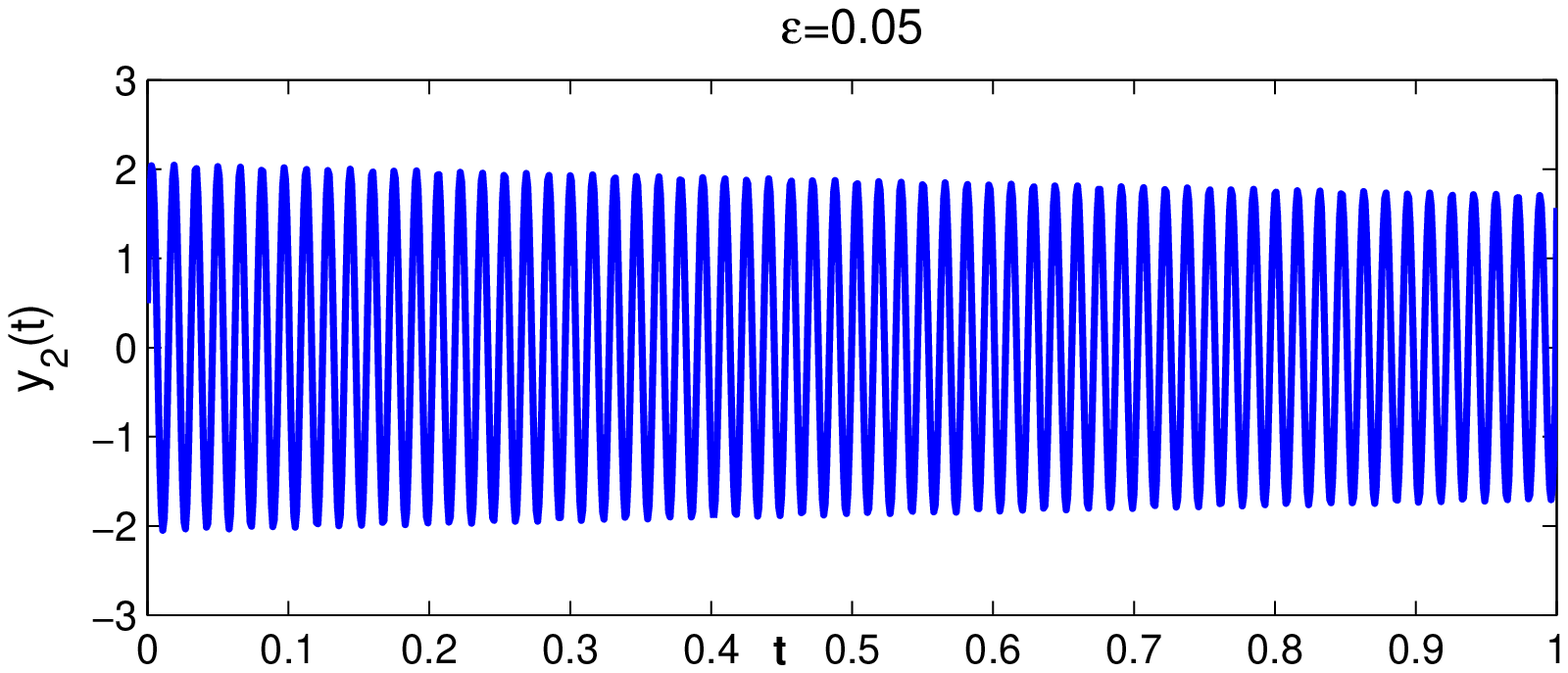,height=4cm,width=17cm}}
\caption{Time evolution of the solutions of (\ref{WODEs})
with $d=2$ for different $\eps$.}\label{fig:00}
\end{figure}

For the global well-posedness of the model problem (\ref{WODEs}), we refer to
 \cite{Hairer1,Hairer2} and references therein.
For simplicity of notations, we will present our methods and comparison for
(\ref{WODEs}) in its simplest case, i.e. $d=1$, as
\begin{equation}\label{WODE}
\left\{
  \begin{split}
    & \eps^2\ddot{y}(t)+\left(\alpha+\frac{1}{\eps^2}\right)y(t)
    +f\left(y(t)\right)=0,\quad t>0,\\
    & y(0)= \phi_1,\quad\dot{y}(0)=\frac{\phi_2}{\eps^2},
  \end{split}
\right.
\end{equation}
where $y=y(t)\in\bC$ is a complex-valued scalar function, $\alpha \geq 0$ is a real
constant, $\phi_{1},\ \phi_{2} \in\bC$, and $f(y):\ \bC \to\bC$.
In particular, in many applications
\cite{Ginibre1,Ginibre2,Glassey1,Glassey2,Machihara, Masmoudi,Simon,Pecher,Segal},
$f(y)$ is taken as the {\sl pure power} nonlinearity as
\begin{equation}\label{power}
f(y)=g(|y|^2)y, \ \hbox{with}\  g(\rho)=\lambda \rho^p\ \hbox{for some}
\ \lambda\in \bR, \ p\in\bN_0:=\bN\cup\{0\}.
\end{equation}
 In addition, if $f$ is taken
 as the pure power nonlinearity (\ref{power}), it is easy to see that
 (\ref{WODE}) conserves the Hamiltonian or total energy, which is given by
\begin{eqnarray}
  E(t)&:=&\eps^2\left|\dot{y}(t)\right|^2
  +\left(\alpha+\frac{1}{\eps^2}\right)\left|y(t)\right|^2
  +F\left(|y(t)|^2\right)\nonumber\\
  &\equiv& \frac{1}{\eps^2} \left|\phi_2\right|^2+
  \left(\alpha+\frac{1}{\eps^2}\right)\left|\phi_1\right|^2
  + F\left(|\phi_1|^2\right):= E(0),\quad t\geq0,\label{energy}
\end{eqnarray}
with $F(\rho)=\int_0^{\rho}g(\rho^\prime)d \rho^\prime$.
Although the numerical methods and their error estimates in this paper
are for the model problem (\ref{WODE}), they can be easily extended to
solve the problem (\ref{WODEs}).  Similar to
the nonlinear Klein-Gordon equation in the
nonrelativistic limit regime \cite{Machihara, Masmoudi},
when $0<\eps\ll 1$, the total energy
$E(t)=O(\eps^{-2})$, i.e., it is unbounded when $\eps\to0$,
with the given initial data in (\ref{WODE}).

We remark here that the model problem (\ref{WODE}) is quite different with
the following  oscillatory second order differential equation arising
from molecular dynamics \cite{Lubich1,Lubich2,Cohen,Cohen1,Sanz,Grimm}
\begin{equation}\label{eq of MD}
\left\{\begin{split}
&\ddot{y}(t)+\frac{1}{\eps^2}y(t)+f(y(t))=0,\qquad t>0,\\
&y(0)=\eps \phi_1, \quad \dot{y}(0)=\phi_2.
\end{split}\right.
\end{equation}
In fact, the above problem (\ref{eq of MD}) propagates waves with wave length and amplitude both at
$O(\varepsilon)$, where the problem (\ref{WODE}) propagates waves with wave length at $O(\varepsilon^2)$
and amplitude at $O(1)$, and thus the oscillation in the problem (\ref{WODE}) is much more oscillating
and wild. In addition, dividing $\eps^2$ on both sides of the model equation (\ref{WODE}), we obtain
\begin{equation}\label{sp1}
\ddot{y}+\frac{\alpha\eps^2+1}{\eps^4}y+\frac{1}{\eps^2}f(y)=0.
\end{equation}
Of course, when $\eps=O(1)$, both (\ref{eq of MD}) and (\ref{sp1}) are perturbations to the harmonic oscillator.
However, in the regime of $0<\eps\ll1$, due to the factor $\frac{1}{\eps^2}$ in front of the nonlinear function,
the nonlinear term in (\ref{sp1}) is not a small perturbation to the harmonic oscillator! Resonance may
occur at time $t=O(1)$.
Another major difference is that the energy of the problem (\ref{eq of MD}) is uniformly bounded for
$\varepsilon\in(0,1]$, where it is unbounded in the problem (\ref{WODE}) when $\varepsilon\to0$.
Different efficient and accurate numerical methods, including finite difference methods \cite{Dong,Duncan},
Gautschi type methods or exponential wave integrators (EWIs) \cite{Lubich1,Cohen,Grimm},
modified impulse methods \cite{Lubich2,Cohen1,Sanz}, modulated Fourier expansion methods \cite{Lubich2,Cohen1,Sanz,Grimm}, heterogeneous multiscale methods \cite{Engquist}, flow averaging \cite{Owhadi},
 Stroboscopic averaging \cite{Cast} and Yong measure approach \cite{Artsein}
have been proposed and analyzed as well as compared
for the problem  (\ref{eq of MD}) in the literatures,
especially in the regime when $0<\varepsilon\ll1$.
In addition, the modulated Fourier expansion has been developed as a powerful analytical tool
for analyzing the oscillating structures of the problem (\ref{eq of MD}) \cite{Cohen,Cohen1,Grimm} and has been
used to design numerical methods for the problem (\ref{eq of MD}) and linear
second-order ODEs with stiff source terms \cite{Cohen,Cohen1,Sanz,Grimm,Condon}.
Based on the results in the literatures \cite{Lubich1,Lubich2,Cohen,Cohen1,Sanz,Grimm},
both the Gautschi type methods and modulated Fourier expansion methods preserve essentially the
total energy and/or oscillatory energy over long times and converge uniformly for $\varepsilon\in (0,1]$
for the problem (\ref{eq of MD}). However, based on the results in \cite{Dong}, all the above numerical methods
do {\sl not} converge uniformly for $\varepsilon\in (0,1]$ for the problem (\ref{WODE})
which usually arise from quantum and plasma physics. In fact,
for existing numerical methods to solve the problem (\ref{WODE}), in order to capture `correctly'
the oscillatory solutions,
one has to restrict the time step $\tau$ in a numerical integrator to be
quite small when $0<\eps\ll 1$.  For instance, as suggested by the
rigorous results in \cite{Dong}, for the frequently used finite
difference (FD) time integrators in the literature \cite{Dong,Duncan,Strauss},
such as energy conservative, semi-implicit and explicit ones, the
meshing strategy requirement (or $\eps$-scalability) is $\tau=O(\eps^3)$ \cite{Dong}.
 Also, a class of trigonometric integrators which solves the
 linear part of (\ref{WODE}) exactly \cite{Dong,Grimm,Lubich1,Lubich2,Sanz,Garcia},
 namely the exponential wave integrators (EWIs), require $\tau=O(\eps^2)$
 for nonlinear problems \cite{Dong}.  In view of that the solutions to
 (\ref{WODE}) are highly oscillatory with  wavelength at $O(\eps^2)$,
 the EWIs could be viewed as the optimal one among the methods which
 integrate the oscillatory problem (\ref{WODE}) directly.

The aim of this paper is to propose and analyze multiscale time
integrators (MTIs) to the problem (\ref{WODE}),
which will converge uniformly for $\eps\in(0,1]$ and thus possess
much better improved $\eps$-scalability than those
classical FD and EWI methods in the regime $0<\eps\ll1$,
by taking into account the sophisticated multiscale structures (see details in (2.2))
in frequency and/or amplitude
of the solutions to (\ref{WODE}).
In our methods, at each time interval,
we adopt an ansatz same as the one used in \cite{Machihara,Masmoudi},
then carry out multiscale decompositions of the solution to
(\ref{WODE}) by either frequency or
frequency and amplitude, and
obtain a coupled equations for two $O(1)$-in-amplitude non-oscillatory components and an $O(\eps^2)$-in-amplitude
oscillatory component. The coupled equations are then discretized by an explicit EWI method
\cite{Grimm,Lubich1,Lubich2} with proper chosen transmission conditions between different time intervals.
Our methods are different from the classical way of applying the modulated Fourier expansion methods
for oscillatory ODEs \cite{Cohen,Cohen1, Condon} in terms of
not only considering the leading order terms but also solving the equation of the remainder which is $O(\eps^2)$
in the pure power nonlinear case  so as  to design a uniformly convergent integrator for any $0<\eps\leq1$.
For the MTIs, we rigorously establish two independent error bounds at $O(\tau^2/\eps^2)$ and $O(\eps^2)$  for $\eps\in(0,1]$
by using the energy method and multiscale analysis \cite{Cai1,Cai2,Dong}.
These two error bounds immediately suggest that
the MTIs converge uniformly with linear convergence rate
at $O(\tau)$ for  $\eps\in(0,1]$ and optimally with quadratic convergence rate at $O(\tau^2)$
in the regimes when either $\eps=O(1)$ or $0<\eps\le \tau$.
Thus, the MTIs offer compelling advantages over those FD and EWI methods for the
problem (\ref{WODE}), especially when $0<\eps\ll1$.

The rest of this paper is organized as follows.  In Section \ref{sec:2}, we present
two multiscale decompositions for the solution of (\ref{WODE}) by either frequency or frequency and amplitude.
Two multiscale time integrators are proposed based on the two multiscale decompositions
and their error bounds are established rigorously when the nonlinearity $f$ satisfies
the power nonlinearity (\ref{power}) and the general nonlinearity (\ref{gauge})
in Sections \ref{sec:3} and \ref{sec:4}, respectively.  In Section \ref{sec:5}, for comparison reasons,  we present the classical FD and EWI discretizations to (\ref{WODE}) and show their rigorous error analysis by paying particular attention
on how error bounds depend on $\eps$ explicitly.  Numerical results are reported in Section \ref{sec:6}.  Finally, some concluding remarks are drawn in Section \ref{sec:7}.  Throughout this paper, we adopt the notation $A\lesssim B$ to represent that there exists a generic constant $C>0$, which is
independent of $\tau$ (or $n$) and $\eps$, such that $|A|\leq CB$.

\section{Multiscale decompositions}
\label{sec:2}

Let $\tau=\Delta t>0$ be the step size, and denote time steps by $t_n=n\tau$ for $n=0,1,\ldots$
In this section, we present multiscale decompositions for the solution of (\ref{WODE})
on the time interval $[t_n, t_{n+1}]$
with given initial data at $t=t_n$ as
\begin{equation}\label{Initial}
y(t_n)=\phi_1^n=O(1),\qquad \dot{y}(t_n)=\frac{\phi_2^n}{\eps^2}=O\left(\frac{1}{\eps^{2}}\right),
\end{equation}
by either frequency or frequency and amplitude.

\subsection{Multiscale decomposition by frequency (MDF)}

Similar to the analytical study of the nonrelativistic limit of the nonlinear Klein-Gordon equation
\cite{Machihara, Masmoudi},
we take an ansatz to the solution $y(t):=y(t_n+s)$ of (\ref{WODE}) on the time interval $[t_n, t_{n+1}]$  with (\ref{Initial}) as
\begin{equation}\label{ansatz}
y(t_n+s)=\fe^{is/\eps^2}z_+^n(s)+\fe^{-is/\eps^2}\overline{z_-^n}(s)
+r^n(s), \qquad 0\leq s\leq \tau.
\end{equation}
Here and after, $\bar{z}$ denotes the complex conjugate of a complex-valued function $z$.
Differentiating (\ref{ansatz}) with respect to $s$, we have
\begin{equation}\label{ansatad}
\dot{y}(t_n+s)=\fe^{is/\eps^2}\left[\dot{z}_+^n(s)+\frac{i}{\eps^2}z_+^n(s)\right]+\fe^{-is/\eps^2}
\left[\overline{\dot{z}_-^n}(s)-\frac{i}{\eps^2}\overline{z_-^n}(s)\right]+
\dot{r}^n(s).
\end{equation}
Plugging (\ref{ansatz}) into (\ref{WODE}), we get
\begin{eqnarray}\label{mdzpm0}
&&\left[2i\dot{z}_+^n(s)+\eps^2\ddot{z}_+^n(s)+\alpha z_+^n(s) \right]\fe^{is/\eps^2}
+\left[-2i\overline{\dot{z}_-^n}(s)+\eps^2\overline{\ddot{z}_-^n}(s)+\alpha \overline{z_-^n}(s)\right]\fe^{-is/\eps^2}
\nonumber\\
&&\quad +\eps^2\ddot{r}^n(s)+\left(\alpha+\frac{1}{\eps^2}\right)r^n(s)+f\left(y(t_n+s)\right)=0,\qquad 0\le s\le \tau.
\end{eqnarray}
Multiplying the above equation by $\fe^{-is/\eps^2}$ and $\fe^{is/\eps^2}$, respectively, we can decompose
the above equation into a coupled system for two $\eps^2$-frequency waves with the unknowns $z_\pm^n(s)$
and the rest frequency waves with the unknown $r^n(s)$ as
\begin{equation}\label{LSADz2}
\left\{
  \begin{split}
  & 2i\dot{z}_\pm^n(s)+\eps^2\ddot{z}_\pm^n(s)+\alpha z_\pm^n(s)+f_\pm\left(z_+^n(s),z_-^n(s)\right)=0,\quad 0<s\leq\tau, \\
  &\eps^2\ddot{r}^n(s)+\left(\alpha+\frac{1}{\eps^2}\right)r^n(s)+f_r\left(z_+^n(s),z_-^n(s),r^n(s);s\right)=0,%   \label{LSARr}
  \end{split}
  \right.
\end{equation}
where
\begin{align}
&f_\pm\left(z_+,z_{-}\right)=\frac{1}{2\pi}\int_0^{2\pi} f\left(z_\pm+\fe^{i\theta}\overline{z_\mp}\right)
d\theta,\label{f_pm def}\\
&f_r\left(z_+,z_-,r;s\right)=f\left(\fe^{is/\eps^2}z_++\fe^{-is/\eps^2}
\overline{z_-}+r\right)-f_+\left(z_+,z_-\right)\fe^{is/\eps^2}-\overline{f_-}\left(z_+,z_-\right)\fe^{-is/\eps^2}.\label{fr def}
\end{align}
In order to find proper initial conditions for the above system (\ref{LSADz2}),
setting $s=0$ in (\ref{ansatz}) and (\ref{ansatad}),
noticing (\ref{Initial}), we obtain
\begin{equation}\label{init123}
\left\{
  \begin{split}
 & z_+^n(0)+\overline{z_-^n}(0)+r^n(0)=y(t_n)=\phi_1^n,\\
 & \frac{i}{\eps^2}\left[z_+^n(0)-\overline{z_-^n}(0)\right]
  +\dot{z}_+^n(0)+\overline{\dot{z}_-^n}(0)
  +\dot{r}^n(0)=\dot{y}(t_n)=\frac{\phi_2^n}{\eps^2}.
  \end{split}
  \right.
\end{equation}
Now we decompose the above initial data so as to: (i) equate $O\left(\frac{1}{\eps^2}\right)$ and $O(1)$ terms
in the second equation of (\ref{init123}), respectively, and (ii) be well-prepared for the first two equations  in (\ref{LSADz2})
when $0<\eps\ll 1$, i.e.
$\dot{z}_+^n(0)$ and $\dot{z}_-^n(0)$ are determined
from the first two equations in (\ref{LSADz2}), respectively, by setting $\eps=0$ and $s=0$ \cite{Cai1,Cai2}:
\begin{equation}\label{FSW-i1}
\left\{
  \begin{split}
&z_+^n(0)+\overline{z_-^n}(0)=\phi_1^n,\qquad i\left[z_+^n(0)-\overline{z_-^n}(0)\right]=\phi_2^n,\\
&2i\dot{z}_\pm^n(0)+\alpha z_\pm^n(0)+f_\pm\left(z_+^n(0),z_-^n(0)\right)=0,\\%\label{FSW-i25}
&r^n(0)=0, \qquad \dot{r}^n(0)+\dot{z}_+^n(0)+\overline{\dot{z}_-^n}(0)=0.%\label{FSW-i3}
\end{split}
  \right.
\end{equation}
Solving (\ref{FSW-i1}), we get the initial data for (\ref{LSADz2}) as
\begin{equation}\label{FSW-i21}
\left\{
  \begin{split}
&z_+^n(0)=\frac{1}{2}\left(\phi_1^n-i\phi_2^n\right), \qquad  z_-^n(0)=\frac{1}{2}\left(\overline{\phi_1^n}-i\ \overline{\phi_2^n}\right),\\
&\dot{z}_\pm^n(0)=\frac{i}{2}\left[\alpha z_\pm^n(0)+f_\pm\left(z_+^n(0),z_-^n(0)\right)\right],\\%\label{FSW-i225}
&r^n(0)=0, \qquad \dot{r}^n(0)=-\dot{z}_+^n(0)-\overline{\dot{z}_-^n}(0).%\label{FSW-i23}
\end{split}
  \right.
\end{equation}
The above decomposition can be called as multiscale decomposition by frequency (MDF).
In fact, it can also be regarded as to decompose slow waves at $\eps^2$-wavelength and fast waves at other wavelengths,
thus it can also be called as fast-slow frequency (FSF) decomposition.

Specifically, for pure power nonlinearity, i.e. $f$ satisfies (\ref{power}), then the above MDF (\ref{LSADz2}) collapses to
\begin{equation}\left\{
\begin{split}
&2i\dot{z}_\pm^n(s)+\eps^2\ddot{z}_\pm^n(s)+\alpha z_\pm^n(s)+g_\pm\left(|z_+^n(s)|^2,|z_-^n(s)|^2\right)z_\pm^n(s)=0,\\
&\eps^2\ddot{r}^n(s)+\left(\alpha+\frac{1}{\eps^2}\right)r^n(s)+g_r\left(z_+^n(s),z_-^n(s),r^n(s);s\right)=0,
\quad 0<s\leq\tau,
\end{split}\right. \label{pLSADz1}
\end{equation}
where
\begin{align}
&g_\pm\left(\rho_+,\rho_{-}\right)=\sum_{\left\langle p_1,p_2,p_3 \right\rangle_0}
    \lambda\left(\rho_{+}+\rho_-\right)^{p_1}(\rho_+\rho_-)^{p_2}(\rho_\mp)^{p_3},\label{g_pm def}\\\
&g_r\left(z_+,z_-,r;s\right)=\sum_{k=1}^p\left(g_k\left(z_+,z_-\right)\fe^{i(2k+1)s/\eps^2}+
\overline{g_k}\left(z_-,z_+\right)\fe^{-i(2k+1)s/\eps^2}\right)+h\left(z_+,z_-,r;s\right),\label{gr def}
\end{align}
with
\begin{eqnarray}
&&g_k\left(z_+,z_-\right)=\lambda(z_+)^{k+1}(z_-)^k\sum_{\left\langle p_1, p_2, p_3 \right\rangle_k}\left(|z_+|^2+|z_-|^2\right)^{p_1}|z_+|^{2p_2}|z_-|^{2p_2+2p_3},\label{g_k def}\\
&&h\left(z_+,z_-,r;s\right)=g\left(|\fe^{is/\eps^2}z_++\fe^{-is/\eps^2}\overline{z_-}+
r|^2\right)\left(\fe^{is/\eps^2}z_++\fe^{-is/\eps^2}
\overline{z_-}+r\right)\nonumber\\
&&\qquad \qquad \qquad \ \ -g\left(|\fe^{is/\eps^2}z_++\fe^{-is/\eps^2}\overline{z_-}|^2\right)\left(\fe^{is/\eps^2}z_++\fe^{-is/\eps^2}
\overline{z_-}\right),\label{h def}
\end{eqnarray}and $\left\langle p_1, p_2, p_3 \right\rangle_k=\left\{p_1, p_2,p_3\in\bN_0 \ \left| \ p_1+2p_2+p_3=p-k,\ p_3=0,1\right.\right\}$ for  $k=0,\ldots,p$.

\subsection{Multiscale decomposition by frequency and amplitude (MDFA)}

Another way to decompose (\ref{mdzpm0}) is to decompose it
into a coupled system for two $\eps^2$-frequency waves at $O(1)$-amplitude with the unknowns $z_\pm^n(s)$
and the rest frequency and amplitude waves with the unknown $r^n(s)$ as
\begin{equation}\left\{
\begin{split}
& 2i\dot{z}_\pm^n(s) +\alpha z_\pm^n(s)+f_\pm\left(z_+^n(s),z_-^n(s)\right)=0,\qquad 0<s\leq\tau,\\
&\eps^2\ddot{r}^n(s)
+\left(\alpha+\frac{1}{\eps^2}\right)r^n(s)+f_r\left(z_+^n(s),z_-^n(s),r^n(s);s\right)+\eps^2u^n(s)=0,
 \end{split}\right.\label{FSWDz31}
\end{equation}
where
\begin{equation}\label{fr_tilde}
u^n(s):=\fe^{is/\eps^2}\ddot{z}_+^n(s)+ \fe^{-is/\eps^2}\overline{\ddot{z}_-^n}(s).
\end{equation}
Similarly, the initial data (\ref{Initial}) can be decomposed as the following for the coupled ODEs
(\ref{FSWDz31})
\begin{equation}\label{FSW-i51}
\left\{
\begin{split}
&z_+^n(0)=\frac{1}{2}\left(\phi_1^n-i\phi_2^n\right), \qquad   z_-^n(0)=\frac{1}{2}
\left(\overline{\phi_1^n}-i\ \overline{\phi_2^n}\right),\\
&r^n(0)=0, \qquad \dot{r}^n(0)=-\dot{z}_+^n(0)-\overline{\dot{z}_-^n}(0),%\label{FSW-i53}
\end{split}
\right.
\end{equation}
with
\begin{equation*}
\dot{z}_\pm^n(0)=\frac{i}{2}\left[\alpha z_\pm^n(0)+f_\pm\left(z_+^n(0),z_-^n(0)\right)\right].
\end{equation*}
In the following, for simplicity of notations, we denote
\begin{equation}\label{fnpms613}
f^n_\pm(s):=f_\pm(z^n_+(s),z_-^n(s)),\quad f_r^n(s):=f_r\left(z_+^n(s),z_-^n(s),r^n(s);s\right).
\end{equation}
The above decomposition can be called as multiscale decomposition by frequency and amplitude (MDFA).
In fact, it can also be regarded as to decompose large amplitude waves at $O(1)$ and small amplitude waves at $O(\eps^2)$,
thus it can also be called as large-small amplitude (LSA) decomposition.

Similarly, for pure power nonlinearity, i.e. $f$ satisfies (\ref{power}), then the above MDFA (\ref{FSWDz31}) collapses to
\begin{equation}\left\{
\begin{split}
&2i\dot{z}_\pm^n(s)+\alpha z_\pm^n(s)+g_\pm\left(|z_+^n(s)|^2,|z_-^n(s)|^2\right)z_\pm^n(s)=0,\quad 0<s\leq\tau,\\
&\eps^2\ddot{r}^n(s)+\left(\alpha+\frac{1}{\eps^2}\right)r^n(s)+
g_r\left(z_+^n(s),z_-^n(s),r^n(s);s\right)+\eps^2u^n(s)=0.
\end{split}\right.\label{pLSADz31}
\end{equation}

After solving the MDF (\ref{LSADz2}) or (\ref{pLSADz1}) with the initial data (\ref{FSW-i21}), or the MDFA (\ref{FSWDz31}) or (\ref{pLSADz31}) with the initial data (\ref{FSW-i51}), we get $z_\pm^n(\tau)$, $\dot{z}_\pm^n(\tau)$, $r^n(\tau)$ and $\dot{r}^n(\tau)$. Then we can reconstruct the solution to (\ref{WODE}) at $t=t_{n+1}$ by setting $s=\tau$ in (\ref{ansatz}) and (\ref{ansatad}), i.e.,
\begin{equation}\label{y n+1}
\left\{
\begin{split}
&y(t_{n+1})=\fe^{i\tau/\eps^2}z_+^n(\tau)+\fe^{-i\tau/\eps^2}\overline{z_-^n}(\tau)
+r^n(\tau):=\phi_1^{n+1},\\
&\dot{y}(t_{n+1})=\frac{1}{\eps^2}\phi_2^{n+1},\\
\end{split}\right.
\end{equation}
with
\[\phi_2^{n+1}:=\fe^{i\tau/\eps^2}\left[\eps^2\dot{z}_+^n(\tau)+i z_+^n(\tau)\right]+\fe^{-i\tau/\eps^2}
\left[\eps^2\overline{\dot{z}_-^n}(\tau)-i\overline{z_-^n}(\tau)\right]+\eps^2\dot{r}^n(\tau).\]

\section{Multiscale time integrators (MTIs) for pure power nonlinearity}
\label{sec:3}

Based on the decomposed system in the pure power nonlinearity case, i.e. the MDFA (\ref{pLSADz31}) or MDF (\ref{pLSADz1}), we propose two multiscale
time integrators (MTI) for solving (\ref{WODE}), respectively.
At each time grid $t=t_n$, we solve the decomposed system (\ref{pLSADz31}) or (\ref{pLSADz1}) by proper integrators within the time interval $[0,\tau]$, and then use (\ref{y n+1}) to reconstruct the solution to (\ref{WODE}) at $t=t_{n+1}$.

\subsection{A MTI based on MDFA}
\label{subsec:LA}
Based on the MDFA (\ref{pLSADz31}), a MTI is designed as follows.

\medskip

\noindent{\sl An exact integrator for $z_\pm^n(s)$ in (\ref{pLSADz31}):}

\medskip

Noting from (\ref{g_pm def}) that $g_\pm\left(\rho_+, \rho_-\right)$
is real-valued, similar to \cite{Bao1,Bao2}, multiplying
the first two equations in (\ref{pLSADz31}) by $\overline{z_\pm^n}(s)$, respectively,
then subtracting from their complex conjugates, we have
\begin{equation}\label{abs_z_c}
\left|z_\pm^n(s)\right|\equiv\left|z_\pm^n(0)\right|,\qquad  0\leq s\leq\tau.
\end{equation}
Therefore, the equations for $z^n_\pm(s)$ in (\ref{pLSADz31}) are exactly integrable, i.e.,
\begin{equation}\label{z_pm exact}
 z_\pm^n(s)=\fe^{{is}\left[g_\pm\left(|z_+^n(0)|^2,|z_-^n(0)|^2\right)
 +\alpha\right]/2}z_\pm^n(0),\qquad 0\leq s\leq\tau.
\end{equation}
Taking $s=\tau$ in (\ref{z_pm exact}), we get
\begin{equation}\label{zpmnt876}
 z_\pm^n(\tau)=\fe^{{i\tau}\left[g_\pm
 \left(|z_+^n(0)|^2,|z_-^n(0)|^2\right)+\alpha\right]/2}z_\pm^n(0).
\end{equation}
Differentiating (\ref{z_pm exact}) with respect to $s$ and then taking $s=0$ or $\tau$, we get
\begin{equation}\label{dz_pm exact}
\left\{
\begin{split}
& \dot{z}_\pm^n(\tau)=\frac{i}{2}\left[g_\pm\left(|z_+^n(0)|^2,|z_-^n(0)|^2\right)+\alpha\right]z_\pm^n(\tau),\\
& \ddot{z}_\pm^n(0)=
 -\frac{1}{4}\left[g_\pm\left(|z_+^n(0)|^2,|z_-^n(0)|^2\right)+\alpha\right]^2z_\pm^n(0),\\
& \ddot{z}_\pm^n(\tau)=
 -\frac{1}{4}\left[g_\pm\left(|z_+^n(0)|^2,|z_-^n(0)|^2\right)+\alpha\right]^2z_\pm^n(\tau).
\end{split}
\right.
\end{equation}
%Thus by letting $s=\tau$ in (\ref{z_pm exact}) and (\ref{dz_pm exact}), we get $z^n_\pm(\tau)$ and its derivatives exactly.

\noindent{\sl An EWI for $r^n(s)$ in (\ref{pLSADz31})}:

\medskip

For the third equation in (\ref{pLSADz31}), we apply the exponential wave integrator (EWI) \cite{Cai2,Gaustchi,Grimm,Dong,Lubich2,Deuflhard,Lubich1,Lubich2,Sanz,Garcia} to solve it, which has favorable properties for solving the second-order oscillatory problems. By applying the variation-of-constant formula to $r^n(s)$, we get
\begin{equation}\label{r gr}
r^n(s)=\frac{\sin(\omega s)}{\omega}\dot{r}^n(0)-\int_0^s\frac{\sin\left(\omega(s-\theta)\right)}{\eps^2\omega}\left[g_r^n(\theta)+\eps^2u^n(\theta)\right]d\theta,
\end{equation}
where
\begin{equation}\label{kappa d}
\omega=\frac{\sqrt{1+\eps^2\alpha}}{\eps^2}=O\left(\frac{1}{\eps^2}\right),\quad g_r^n(\theta):=g_r(z_+^n(\theta),z_-^n(\theta),r^n(\theta);\theta).%\qquad\tilde{f}_r^n(\theta):=\tilde{f}_r(z_+^n(\theta),z_-^n(\theta),r^n(\theta);\theta).
\end{equation}
Taking $s=\tau$ in (\ref{r gr}), we get
\begin{equation}\label{r tau gr}
r^n(\tau)=\frac{\sin(\omega \tau)}{\omega}\dot{r}^n(0)-\int_0^\tau\frac{\sin\left(\omega(\tau-\theta)
\right)}{\eps^2\omega}\left[g_r^n(\theta)+\eps^2u^n(\theta)\right]d\theta.
\end{equation}
Differentiating (\ref{r gr}) with respect to $s$ and then taking $s=\tau$, we get
\begin{equation}\label{dr tau gr}
\dot{r}^n(\tau)=\cos(\omega \tau)\dot{r}^n(0)-\int_0^\tau\frac{\cos\left(\omega(\tau-\theta)
\right)}{\eps^2}\left[g_r^n(\theta)+\eps^2u^n(\theta)\right]d\theta.
\end{equation}

Plugging (\ref{gr def}) into (\ref{r tau gr}) and (\ref{dr tau gr}), we find
\begin{equation}\label{VCFr}
\left\{
\begin{split}
&r^n(\tau)=\frac{\sin(\omega \tau)}{\omega}\dot{r}^n(0)-
\sum_{k=1}^p\left[I^n_{k,+}+\overline{I^n_{k,-}}\right]-J^n,\\
&\dot{r}^n(\tau)=\cos(\omega \tau)\dot{r}^n(0)-\sum_{k=1}^p\left[\dot{I}_{k,+}^n+
\overline{\dot{I}^n_{k,-}}\right]-\dot{J}^n,%\label{VCFdr}
\end{split}
\right.
\end{equation}
where
\begin{equation}\label{Ikpmn965}
  \left\{
  \begin{split}
&I_{k,\pm}^n= \int_0^\tau\frac{\sin(\omega(\tau-
\theta))}{\eps^2\omega}\fe^{i(2k+1)\theta/\eps^2}g_{k,\pm}^n(\theta) d\theta,\\
&J^n= \int_0^\tau\frac{\sin(\omega(\tau-\theta))}{\eps^2\omega}
\left[h^n(\theta)+\eps^2u^n(\theta)\right]d\theta,\\
&\dot{I}_{k,\pm}^n= \int_0^\tau\frac{\cos(\omega(\tau-\theta))}
{\eps^2}\fe^{i(2k+1)\theta/\eps^2}g_{k,\pm}^n(\theta) d\theta,\\
&\dot{J}^n= \int_0^\tau\frac{\cos(\omega(\tau-\theta))}{\eps^2}
\left[h^n(\theta)+\eps^2u^n(\theta)\right]d\theta,
  \end{split}
  \right.
\end{equation}
with
\begin{equation}\label{gkpm321}
g_{k,\pm}^n(\theta):=g_{k}(z_\pm^n(\theta),z_\mp^n(\theta)),\qquad h^n(\theta):=h\left(z_+^n(\theta),z_-^n(\theta),r^n(\theta);\theta\right).
\end{equation}
In order to have an explicit integrator and achieve uniform error bounds, we approximate the integral terms $I_{k,\pm}^n$ and $\dot{I}_{k,\pm}^n$ in (\ref{Ikpmn965}) by a quadrature in the Gautschi's type \cite{Gaustchi} as the following
which was discussed and used in \cite{Cai2,Dong}
\begin{equation}\label{quad6543}
\left\{
\begin{split}
&I_{k,\pm}^n\approx\int_0^\tau\frac{\sin(\omega(\tau-
\theta))}{\eps^2\omega}\fe^{i(2k+1)\theta/\eps^2}\left[g_{k,\pm}^n(0)+\theta \dot{g}_{k,\pm}^n(0)\right] d\theta\\
&\qquad\  =p_kg_{k,\pm}^n(0)+q_k\dot{g}_{k,\pm}^n(0),\\
&\dot{I}_{k,\pm}^n\approx\int_0^\tau\frac{\cos(\omega(\tau-\theta))}
{\eps^2}\fe^{i(2k+1)\theta/\eps^2} \left[g_{k,\pm}^n(0)+\theta \dot{g}_{k,\pm}^n(0)\right]d\theta\\
&\qquad\ =\dot{p}_kg_{k,\pm}^n(0)+\dot{q}_k\dot{g}_{k,\pm}^n(0),\\
\end{split}
\right.\end{equation}
where (their detailed explicit formulas are shown in Appendix A)
\begin{align*}\label{pkqk}
&p_k=\int_0^\tau\frac{\sin(\omega(\tau-\theta))}{\eps^2\omega}\fe^{i(2k+1)\theta/\eps^2}d\theta,\quad q_k=\int_0^\tau \frac{\sin(\omega(\tau-\theta))}{\eps^2\omega}\fe^{i(2k+1)\theta/\eps^2}\theta d\theta,\\
&\dot{p}_k=\int_0^\tau\frac{\cos(\omega(\tau-\theta))}{\eps^2}\fe^{i(2k+1)\theta/\eps^2}d\theta,\quad \dot{q}_k=\int_0^\tau \frac{\cos(\omega(\tau-\theta))}{\eps^2}\fe^{i(2k+1)\theta/\eps^2}\theta d\theta.\ \
\end{align*}
In addition, approximating $J^n$ and $\dot{J}^n$ in (\ref{Ikpmn965}) by the standard single step trapezoidal rule and noticing
$h^n(0)=0$, we get
\begin{equation}\label{quad64321}
\left\{
\begin{split}
&J^n\approx \frac{\tau}{2} \frac{\sin(\omega\tau)}{\varepsilon^2\omega}\left[h^n(0)+\varepsilon^2 u^n(0)\right]
=\frac{\tau}{2} \frac{\sin(\omega\tau)}{\omega}u^n(0),\\
&\dot{J}^n\approx \frac{\tau}{2}\left[\frac{\cos(\omega\tau)}{\varepsilon^2}\left(h^n(0)+\varepsilon^2 u^n(0)\right)
+\frac{1}{\varepsilon^2}\left(h^n(\tau)+\varepsilon^2 u^n(\tau)\right)\right].
\end{split}
\right.\end{equation}
Plugging (\ref{quad6543}), (\ref{quad64321}) and (\ref{Ikpmn965}) into (\ref{VCFr}) and noticing
$h^n(0)=0$, we obtain
\begin{equation}\label{quad Ik J}
\left\{
\begin{split}
&r^n(\tau)\approx-\sum_{k=1}^p\left[p_k g_{k,+}^n(0)+q_k \dot{g}_{k,+}^n(0)+\overline{p_k g_{k,-}^n(0)}+
\overline{q_k \dot{g}_{k,-}^n(0)}\right]\\
&\ \qquad\quad+\frac{\sin(\omega\tau)}{\omega}\left[\dot{r}^n(0)-\frac{\tau}{2}u^n(0)\right],\\
&\dot{r}^{n}(\tau)\approx-\sum_{k=1}^p\left[\dot{p}_k g_{k,+}^n(0)+\dot{q}_k \dot{g}_{k,+}^n(0)+
\overline{\dot{p}_k g_{k,-}^n(0)}+\overline{\dot{q}_k \dot{g}_{k,-}^n(0)}\right] \\
&\ \qquad\quad+\cos(\omega\tau)\left[\dot{r}^n(0)-\frac{\tau}{2}u^n(0)\right]-\frac{\tau}{2}
\left[\frac{h^n(\tau)}{\eps^2}+u^n(\tau)\right].
\end{split}
\right.
\end{equation}

\medskip

\noindent{\sl Detailed numerical scheme}

\medskip

 For $n=0,1,\ldots,$ let $y^n$ and $\dot{y}^n$ be the approximations of $y(t_n)$ and $\dot{y}(t_n)$, $z_{\pm}^{n+1}$, $\dot{z}_{\pm}^{n+1}$, $\ddot{z}_{\pm}^{n+1}$, $r^{n+1}$ and $\dot{r}^{n+1}$ be the approximations of $z_{\pm}^n(\tau)$, $\dot{z}_{\pm}^n(\tau)$, $\ddot{z}_{\pm}^n(\tau)$, $r^n(\tau)$ and $\dot{r}^n(\tau)$, respectively, where $z_{\pm}^n(s)$ and $r^n(s)$ are the solutions to the system (\ref{pLSADz31}) with initial data  (\ref{FSW-i51}). Choosing $y^0=y(0)=\phi_1$ and $\dot{y}^0=\dot{y}(0)=\eps^{-2}\phi_2$, for $n=0,1,\ldots$, $y^{n+1}$ and $\dot{y}^{n+1}$ are updated as follows:
\begin{equation}\label{IFSW}
\left\{
\begin{split}
&y^{n+1} = \fe^{i\tau/\eps^2}z_+^{n+1}+\fe^{-i\tau/\eps^2}\overline{z_-^{n+1}}+r^{n+1},\\
&\dot{y}^{n+1}  =\fe^{i\tau/\eps^2}\left(\dot{z}_+^{n+1}+\frac{i}{\eps^2}z_+^{n+1}\right)+ \fe^{-i\tau/\eps^2}\left(\overline{\dot{z}_-^{n+1}}-\frac{i}{\eps^2}\overline{z_-^{n+1}}\right)+\dot{r}^{n+1},
\end{split}
\right.
\end{equation}
where
\begin{numcases}
\ z_\pm^{n+1}=\fe^{i\mu_\pm\tau}z_\pm^{(0)},\qquad\dot{z}_\pm^{n+1}=i\mu_\pm z_\pm^{n+1},\qquad \ddot{z}_\pm^{n+1}=-(\mu_\pm)^2z_\pm^{n+1},\nonumber\\
r^{n+1}=\frac{\sin(\omega\tau)}{\omega}\left(\dot{r}^{(0)}-\frac{\tau}{2}u^{(0)}\right)
-\sum_{k=1}^p\left[p_kg_{k,+}^{(0)}+q_k\dot{g}_{k,+}^{(0)}+\overline{p_kg_{k,-}^{(0)}}+
\overline{q_k\dot{g}_{k,-}^{(0)}}\right],\nonumber\\
%&\ \qquad\quad+,\\
\dot{r}^{n+1}=-\sum_{k=1}^p\left[\dot{p}_kg_{k,+}^{(0)}+\dot{q}_k\dot{g}_{k,+}^{(0)}+
\overline{\dot{p}_kg_{k,-}^{(0)}}+\overline{\dot{q}_k\dot{g}_{k,-}^{(0)}}\right] \nonumber\\
\ \qquad\quad+\cos(\omega\tau)\left(\dot{r}^{(0)}-\frac{\tau}{2}u^{(0)}\right)-
\frac{\tau}{2}\left(\frac{h^{n+1}}{\eps^2}+u^{n+1}\right),\label{FS_r}\\
u^{n+1}=\fe^{i\tau/\eps^2}\ddot{z}_+^{n+1}+\fe^{-i\tau/\eps^2}\overline{\ddot{z}_-^{n+1}},\nonumber\\
h^{n+1}=g(|y^{n+1}|^2)y^{n+1}-g\left(\left|y^{n+1}-
r^{n+1}\right|^2\right)\left(y^{n+1}-r^{n+1}\right),\nonumber
\end{numcases}
with
\begin{equation}\label{F_ini1}
\left\{
\begin{split}
&z_+^{(0)}=\frac{y^n-i\eps^2\dot{y}^{n}}{2},\quad
 z_-^{(0)}=\frac{\overline{y^n}-i\eps^2\overline{\dot{y}^{n}}}{2},\quad\dot{z}_\pm^{(0)}=i\mu_\pm z_\pm^{(0)},\\
&\dot{r}^{(0)}=-\dot{z}_+^{(0)}-\overline{\dot{z}_-^{(0)}},\qquad u^{(0)}=-(\mu_+)^2z_+^{(0)}-(\mu_-)^2\overline{z_-^{(0)}},\\
&\mu_\pm=\frac{1}{2}g_{\pm}\left(\left|z_+^{(0)}\right|^2,\left|z_-^{(0)}\right|^2\right)+\frac{\alpha}{2},\qquad
  g_{k,\pm}^{(0)}=g_k\left(z_\pm^{(0)},z_\mp^{(0)}\right),\\
&\dot{g}_{k,\pm}^{(0)}=\frac{d}{ds}\left[g_{k}\left(z_+(s),z_-(s)\right)\right]\Big|_{\left\{z_\pm=z_\pm^{(0)},\ \dot{z}_\pm=\dot{z}_\pm^{(0)}\right\}},\qquad k=1,\ldots,p.
\end{split}
\right.
\end{equation}

We call the proposed numerical integrator (\ref{IFSW}) with (\ref{FS_r}) as a multiscale time integrator  based on MDFA which is abbreviated as MTI-FA in short. Clearly, MTI-FA is fully explicit, and easy to implement in practice. In fact, in this scheme, at the beginning of each time interval $[t_n,t_{n+1}]$, we decompose the numerical solutions $y^n$ and $\dot{y}^n$ to specify the initial conditions of the system (\ref{FSWDz31});
  then we solve the decomposed system numerically; at the end of each time interval, we reconstruct the approximations $y^{n+1}$ and $\dot{y}^{n+1}$ from the numerical solutions to (\ref{FSWDz31}). Therefore, at each time step, the algorithm proceeds as decomposition-solution-reconstruction.

\subsection{Another MTI based on MDF}
Based on the MDF (\ref{pLSADz1}), we propose another MTI as follows. Since the system (\ref{pLSADz1}) consists of three second-order oscillatory problems, so we use EWIs to solve it.
\medskip

\noindent{\sl An EWI for (\ref{pLSADz1}):}

\medskip
By applying the variation-of-constant formula to the first two equations in (\ref{LSADz2}), we have
\begin{equation}\label{2VCFz}
z_\pm^n(s)=a(s)z_\pm^n(0)+\eps^2b(s)\dot{z}_\pm^n(0)
-\int_0^{s}b(s-\theta)f_\pm^n(\theta) d\theta,\qquad 0\le s\le \tau,
\end{equation}
where
\begin{numcases}
\ a(s):=\frac{\lambda_+\fe^{is\lambda_-}-\lambda_-\fe^{is\lambda_+}}{\lambda_+-\lambda_-},\quad b(s):=i\frac{\fe^{is\lambda_+}-
\fe^{is\lambda_-}}{\eps^2(\lambda_--\lambda_+)}, \qquad 0\le s\le \tau,\nonumber\\
 \lambda_{+}=-\frac{1}{\eps^2}\left(1+\sqrt{1+\alpha\eps^2}\right)=O\left(\frac{1}{\eps^2}\right),\label{asbslbd}\\
 \lambda_{-}=-\frac{1}{\eps^2}\left(1-\sqrt{1+\alpha\eps^2}\right)=O(1). \nonumber
\end{numcases}
Taking $s=\tau$ in (\ref{2VCFz}), we get
\begin{equation}\label{2VCFz tau}
z_\pm^n(\tau)=a(\tau)z_\pm^n(0)+\eps^2b(\tau)\dot{z}_\pm^n(0)-\int_0^{\tau}b(\tau-\theta)f_\pm^n(\theta) d\theta.
\end{equation}
Differentiating (\ref{2VCFz}) with respect to $s$ and then taking $s=\tau$, we get
\begin{equation}\label{2VCFdz}
\dot{z}_\pm^n(\tau)=\dot{a}(\tau)z_\pm^n(0)+\eps^2\dot{b}(\tau)\dot{z}_\pm^n(0)-\int_0^{\tau}\dot{b}(\tau-\theta)f_\pm^n(\theta)d\theta,
\end{equation}
where
\begin{equation*}
\dot{a}(s)=i\lambda_+\lambda_-\frac{\fe^{is\lambda_-}-\fe^{is\lambda_+}}{\lambda_+-\lambda_-},\qquad \dot{b}(s)=\frac{\lambda_+\fe^{is\lambda_+}-\lambda_-\fe^{is\lambda_-}}{\eps^2(\lambda_+-\lambda_-)},
\qquad 0\le s\le \tau.
\end{equation*}
Then approximating the integral terms in (\ref{2VCFz tau}) and (\ref{2VCFdz}) by the Gautschi's type quadrature similar as (\ref{quad6543}), we have
\begin{equation}\label{gen F appr}
\left\{
\begin{split}
&z_\pm^n(\tau)\approx a(\tau)z_\pm^n(0)+\eps^2b(\tau)\dot{z}_\pm^n(0)-c(\tau)f_\pm^n(0)-d(\tau)\dot{f}_\pm^n(0),\\
&\dot{z}_\pm^n(\tau)\approx\dot{a}(\tau)z_\pm^n(0)+\eps^2\dot{b}(\tau)\dot{z}_\pm^n(0)
-\dot{c}(\tau)f_\pm^n(0)-\dot{d}(\tau)\dot{f}_\pm^n(0),
\end{split}\right.
\end{equation}
where (their detailed explicit formulas are shown in Appendix A)
\begin{align*}
&c(\tau):=\int^\tau_0 b(\tau-\theta)d\theta,\quad d(\tau):=\int^\tau_0  b(\tau-\theta){}\theta d\theta,\\
&\dot{c}(\tau):=\int^\tau_0 \dot{b}(\tau-\theta)d\theta,\quad \dot{d}(\tau):=\int^\tau_0\dot{b}(\tau-\theta){}\theta d\theta.
\end{align*}
Now, substituting
$$f_\pm^n(s)=g_\pm(|z^n_+(s)|^2,|z^n_-(s)|^2)z_\pm^n(s)$$
into (\ref{gen F appr}), we obtain the approximations to $z_\pm^n(\tau)$ and $\dot{z}_\pm^n(\tau)$.

As for the last equation in (\ref{pLSADz1}), again by the variation-of-constant formula and noticing (\ref{gr def}), we can derive the integral forms for $r^n(\tau)$ and $\dot{r}^n(\tau)$ same as (\ref{VCFr}) but without $u^n$ terms defined in $J^n$ and $\dot{J}^n$. Then the rest approximations are similar to (\ref{quad Ik J}).

\medskip

\noindent{\sl Detailed numerical scheme}

\medskip

Following the same notations introduced in subsection \ref{subsec:LA}, choosing $y^0=y(0)=\phi_1$ and $\dot{y}^0=\dot{y}(0)=\eps^{-2}\phi_2$, for $n=0,1,\ldots,$ $y^{n+1}$ and $\dot{y}^{n+1}$ are updated in the same way as (\ref{IFSW})-(\ref{F_ini1}) except that
\begin{numcases}
\ z_\pm^{n+1}=a(\tau)z_\pm^{(0)}+\eps^2b(\tau)\dot{z}_\pm^{(0)}-
  c(\tau)f_\pm\left(z_+^{(0)},z_-^{(0)}\right)-d(\tau)\dot{f}_\pm^{(0)},\nonumber\\
  \dot{z}_\pm^{n+1}=\dot{a}(\tau)z_\pm^{(0)}+\eps^2\dot{b}(\tau)\dot{z}_\pm^{(0)}
  -\dot{c}(\tau)f_\pm\left(z_+^{(0)},z_-^{(0)}\right)-\dot{d}(\tau)\dot{f}_\pm^{(0)},\nonumber\\
\ r^{n+1}=\frac{\sin(\omega\tau)}{\omega}\dot{r}^{(0)}
-\sum_{k=1}^p\left[p_kg_{k,+}^{(0)}+q_k\dot{g}_{k,+}^{(0)}+\overline{p_kg_{k,-}^{(0)}}+
\overline{q_k\dot{g}_{k,-}^{(0)}}\right],\label{F s}\\
 \dot{r}^{n+1}= \cos(\omega\tau)\dot{r}^{(0)}-\frac{\tau}{2\eps^2}h^{n+1}
-\sum_{k=1}^p\left[\dot{p}_kg_{k,+}^{(0)}+\dot{q}_k\dot{g}_{k,+}^{(0)}+
\overline{\dot{p}_kg_{k,-}^{(0)}}+\overline{\dot{q}_k\dot{g}_{k,-}^{(0)}}\right],\nonumber\\
\dot{f}_{\pm}^{(0)}=\frac{d}{ds}\left[f_\pm(z_+(s),z_-(s))\right]\Big|_{\left\{z_\pm=z_\pm^{(0)},\ \dot{z}_\pm=\dot{z}_\pm^{(0)}\right\}}.\nonumber
\end{numcases}

Again, we call the proposed numerical integrator (\ref{IFSW}) with (\ref{F s})
 as a multiscale time integrator  based on MDF which is abbreviated as MTI-F in short.
 Clearly, MTI-F is fully explicit, and easy to implement in practice.

\subsection{Error estimates of MTIs for pure power nonlinearity}\label{MDFA_est}
Here, we shall give the convergence result of the proposed MTIs for the pure power nonlinearity case. In order to obtain rigorous error estimates, we assume that the exact solution $y(t)$ to (\ref{WODE}) satisfies the following assumptions
\begin{equation}\label{assumption}
y(t)\in C^2(0,T),\quad\mbox{and}\quad \left\|\frac{d^m}{dt^m}y(t)\right\|_{L^\infty(0,T)}\lesssim\frac{1}{\eps^{2m}},\quad m=0,1,2,
\end{equation}
for $0<T<T^\ast$ with $T^\ast$ the maximum existence time. Denoting
\begin{equation}\label{C0}
C_0:=\max\left\{\|y\|_{L^{\infty}(0,T)},\ \eps^2\|\dot{y}\|_{{L^{\infty}(0,T)}},\ \eps^4\|\ddot{y}\|_{{L^{\infty}(0,T)}}\right\},
\end{equation}
and the error functions as
\begin{equation}\label{error_fun}
  e^n:=y(t_n)-y^n,\quad \dot{e}^n:=\dot{y}(t_n)-\dot{y}^n,
\end{equation}
then we have the following error estimates for MTI-FA
 (see detailed proof in Appendix \ref{Ap proof MTI_FA}) and MTI-F (see detailed proof in Appendix \ref{Ap proof MTI_F}).

\begin{theorem}[Error bounds of MTI-FA]\label{tm1}
For numerical integrator MTI-FA, i.e. (\ref{IFSW}) with
(\ref{FS_r}), under the assumption (\ref{assumption}),
there exits a constant $\tau_0>0$ independent of $\eps$
and $n$, such that for any $0<\eps\leq 1$
\begin{align}
  &|e^n|+\eps^2|\dot{e}^n|\lesssim\frac{\tau^2}{\eps^2},\qquad
  |e^n|+\eps^2|\dot{e}^n|\lesssim \eps^2,\qquad 0<\tau\leq\tau_0,\label{FSWDerror1}\\
  &|y^n|\leq C_0+1,\qquad |\dot{y}^n|\leq\frac{C_0+1}{\eps^2},
  \qquad 0\leq n \leq \frac{T}{\tau}.\label{FSWDbound}
\end{align}
Thus by taking the minimum of two error bounds for $0<\eps\le 1$, we have a uniform error bound as
\begin{equation}\label{FSWDerror2}
|e^n|+\eps^2|\dot{e}^n|\lesssim \min_{0<\eps\le 1}\left\{\frac{\tau^2}{\eps^2},\eps^2\right\}\lesssim
\tau,\quad 0\leq n \leq \frac{T}{\tau},\quad 0<\tau\leq\tau_0.
\end{equation}
\end{theorem}

\begin{theorem}[Error bounds of MTI-F]\label{tm2}
For the numerical integrator MTI-F, i.e. (\ref{IFSW}) with (\ref{F s}),
under the assumption (\ref{assumption}), there exists a constant $\tau_0>0$
independent of $\eps$ and $n$, such that for any $0<\eps\leq 1$
\begin{align}
  &|e^n|+\eps^2|\dot{e}^n|\lesssim\frac{\tau^2}{\eps^2},\qquad
  |e^n|+\eps^2|\dot{e}^n|\lesssim \tau^2+\eps^2,\qquad 0<\tau\leq\tau_0,\label{LSADerror1}\\
  &|y^n|\leq C_0+1,\qquad |\dot{y}^n|\leq\frac{C_0+1}{\eps^2},
  \qquad 0\leq n \leq \frac{T}{\tau}.\label{LSADbound}
\end{align}
Thus by taking the minimum of two error bounds for $0<\eps\le 1$, we have a uniform error bound as
\begin{equation}\label{LSADerror2}
|e^n|+\eps^2|\dot{e}^n|\lesssim\min_{0<\eps\le 1}\left\{\frac{\tau^2}{\eps^2}, \ \tau^2+\eps^2\right\}\lesssim \tau,\quad 0\leq n \leq \frac{T}{\tau},\quad 0<\tau\leq\tau_0.
\end{equation}
\end{theorem}

\begin{remark}
If $\phi_1$, $\phi_2\in {\mathbb R}$, $y:=y(t)$ is a real-valued function and
$f(y):\ {\mathbb R}\to {\mathbb R}$  in (\ref{WODE}), then it is easy to see
that $z_-^n(s) =z_+^n(s)$ for $0\le s\le \tau$ in (\ref{ansatz}) from (\ref{LSADz2}) and (\ref{FSW-i21}),
and (\ref{FSWDz31}) and (\ref{FSW-i51}) for MDF and MDFA, respectively.
Thus the multiscale decompositions MDF and MDFA and their numerical integrators
MTI-F and MTI-FA as well as their error estimates are still valid and can be simplified.
We omit the details here for brevity.
\end{remark}

\begin{remark}
The two MTIs for the problem (\ref{WODE}), i.e. MTI-FA and MTI-F, are completely different
with the modulated Fourier expansion methods proposed in the literatures \cite{Lubich1,Lubich2,Cohen,Cohen1,Sanz,Grimm} for the problem (\ref{eq of MD}) in the following aspects. (i) As stated in Section 1, they are used to solve second order ODEs
with different oscillatory behavior in the solutions. (ii) In our MTIs, we adapt the expansion (\ref{ansatz}) at each time
interval $[t_n,t_{n+1}]$ and update its initial data  via proper transmission conditions between different time intervals, and the decoupled system consists of only three equations including two equations for the two leading
frequencies and one equation for reminder. However, in the modulated Fourier expansion methods,
it expands the solution only once at $t=0$ and up to finite terms with increasing frequencies by dropping the reminder, and thus the decoupled system consists of finite number of equations.
(iii) Our MTIs are uniformly accurate
for $\varepsilon\in (0,1]$ for the problem (\ref{WODE}) and the error only depends
on the time step and is independent of $\varepsilon$ and the terms in the expansion (\ref{ansatz}).
However, if the modulated Fourier expansion methods are applied to the problem (\ref{WODE}),
they are usually asymptotic preserving methods instead of uniformly accurate methods.
In addition, the errors depend on time step, $\varepsilon$ and
the number of terms used in the expansion. If high accuracy is needed, one needs to use many terms in the expansion and thus they might be expensive. (iv) Our MTIs work for the regimes when $\varepsilon$ is small, large and intermediate;
where the modulated Fourier expansion methods only work for the regime when $\varepsilon$ is small.
\end{remark}

\section{Multiscale time integrators (MTIs) for general nonlinearity}
\label{sec:4}
In this section, based on the MDFA (\ref{FSWDz31}) or MDF (\ref{LSADz2}) for a general gauge invariant nonlinearity $f(y)$ in (\ref{WODE}), we propose two multiscale time integrators (MTIs) for solving (\ref{WODE}). We will adopt the notations introduced in section \ref{sec:3}.

\subsection{A MTI based on MDFA}\label{gen MTI-F}
Based on the MDFA (\ref{FSWDz31}), we propose a MTI.

Integrating the first two equations for $z_\pm^n(s)$ in (\ref{FSWDz31}) over  $[0,\tau]$, we get
\begin{equation}\label{MTI-F z_pm}
z_\pm^n(\tau)=\fe^{\frac{i\alpha}{2} \tau}z_\pm^n(0)+\frac{i}{2}\int_0^\tau\fe^{\frac{i\alpha}{2}(\tau-s)}f_\pm^n(s)ds.
\end{equation}
Similar to (\ref{quad6543}), we approximate the integral term by a quadrature in the Gautschi's type, i.e.,
\begin{align}
z_\pm^n(\tau)&\approx\fe^{\frac{i\alpha}{2} \tau}z_\pm^n(0)+\frac{i}{2}\int_0^\tau\fe^{\frac{i\alpha}{2}(\tau-s)}
\left[f_\pm^n(0)+s\dot{f}_\pm^n(0)\right]ds\nonumber\\
&=\fe^{\frac{i\alpha}{2} \tau}z_\pm^n(0)+\beta_1f_\pm^n(0)+\beta_2\dot{f}_\pm^n(0),\label{MTIap1}
\end{align}
where
\begin{equation*}
\beta_1=\frac{i}{2\alpha}\left(\fe^{\frac{i\alpha}{2}\tau}-1\right),\quad\quad \beta_2=\frac{1}{2\alpha^2}\left(2\fe^{\frac{i\alpha}{2}\tau}-i\alpha\tau-2\right).
\end{equation*}
Taking $s=\tau$ in the first two equations in (\ref{FSWDz31}), we find
\begin{equation}
\dot{z}_\pm^n(\tau)=\frac{i\alpha}{2}z_\pm^n(\tau)+\frac{i}{2}f_\pm^n(\tau).
\end{equation}

For the third equation in (\ref{FSWDz31}), we apply the exponential wave integrator (EWI) to solve it. Using the variation-of-constant formula, we obtain
\begin{equation}\label{gen r tau}
\left\{
\begin{split}
&r^n(\tau)=\frac{\sin(\omega \tau)}{\omega}\dot{r}^n(0)-
\int_0^\tau\frac{\sin\left(\omega(\tau-\theta)
\right)}{\eps^2\omega}\left[f_r^n(\theta)+\eps^2u^n(\theta)\right]d\theta,\\
&\dot{r}^n(\tau)=\cos(\omega \tau)\dot{r}^n(0)-
\int_0^\tau\frac{\cos\left(\omega(\tau-\theta)\right)}
{\eps^2}\left[f_r^n(\theta)+\eps^2u^n(\theta)\right]d\theta,
\end{split}
\right.
\end{equation}
To have an explicit integrator and achieve  uniform error bounds,
we approximate the two integral terms in (\ref{gen r tau})
 by quadratures intended to preserve different
scales produced by the two integrands. In order to do so, due to that $f_r^n(0)\ne0$,
we introduce two linear interpolations for $f^n_r(\theta)$ on the interval $[0,\tau]$ as
\begin{equation}\label{interpolation}
l_1^n(\theta)=\frac{\tau-\theta}{\tau}f_r^n(0),\qquad l_2^n(\theta)=\frac{\theta}{\tau}f_r^n(\tau)+\frac{\tau-\theta}{\tau}f_r^n(0),
 \qquad 0\le \theta\le \tau.
\end{equation}
In addition, differentiating the first two equations in (\ref{FSWDz31}) with respect to $s$ and then taking $s=0$ or $\tau$, we get
\begin{equation}
\ddot{z}_\pm^n(0)=\frac{i\alpha}{2}\dot{z}_\pm^n(0)+\frac{i}{2}\dot{f}_\pm^n(0),\qquad
\ddot{z}_\pm^n(\tau)=\frac{i\alpha}{2}\dot{z}_\pm^n(\tau)+\frac{i}{2}\dot{f}_\pm^n(\tau).
\end{equation}
Combing the above and applying the trapezoidal rule, we have
\begin{align}
&\int_0^\tau\frac{\sin(\omega(\tau-\theta))}{\eps^2\omega}
\left[f_r^n(\theta)+\eps^2u^n(\theta)\right]d\theta\nonumber\\
&=\int_0^\tau\frac{\sin(\omega(\tau-\theta))}{\eps^2\omega}
\left[f_r^n(\theta)-l_1^n(\theta)+\eps^2u^n(\theta)\right]d\theta
+\int_0^\tau\frac{\sin(\omega(\tau-\theta))}{\eps^2\omega}l_1^n(\theta)d\theta\nonumber\\
&\approx\frac{\tau\sin(\omega\tau)}{2\omega}u^n(0)+\gamma_1f_r^n(0), \label{quad123}
\end{align}
\begin{align}\label{f_pm appr}
& \int_0^\tau\frac{\cos(\omega(\tau-\theta))}{\eps^2}
\left[f_r^n(\theta)+\eps^2u^n(\theta)\right]d\theta\nonumber\\
&=\int_0^\tau\frac{\cos(\omega(\tau-\theta))}{\eps^2}
\left[f_r^n(\theta)-l_2^n(\theta)+\eps^2u^n(\theta)\right]d\theta
+\int_0^\tau\frac{\cos(\omega(\tau-\theta))}{\eps^2}l_2^n(\theta)d\theta\nonumber\\
&\approx\frac{\tau}{2}\left[\cos(\omega\tau)u^n(0)+u^n(\tau)\right]
+\gamma_2f_r^n(0)+\gamma_3f_r^n(\tau),
\end{align}
where
\begin{equation*}
\gamma_1=\frac{1-\cos(\omega\tau)}{\eps^2\omega^2},
\quad\gamma_2=\frac{\cos(\omega\tau)+\omega\tau\sin(\omega\tau)-1}{\eps^2\omega^2\tau},\quad
 \gamma_3=\frac{1-\cos(\omega\tau)}{\eps^2\omega^2\tau}.
\end{equation*}
Plugging (\ref{quad123}) and (\ref{f_pm appr}) into (\ref{gen r tau}),  we obtain
\begin{equation}\left\{
\begin{split}
&r^n(\tau)\approx\frac{\sin(\omega\tau)}{\omega}\left[\dot{r}^n(0)
-\frac{\tau}{2}u^n(0)\right]-\gamma_1f_r^n(0),\\
&\dot{r}^n(\tau)\approx\cos(\omega\tau)\left[\dot{r}^n(0)-\frac{\tau}{2}u^n(0)
\right]-\frac{\tau}{2}u^n(\tau)-\gamma_2f_r^n(0)
-\gamma_3f_r^n(\tau),
\end{split}\right.
\end{equation}
where
\begin{equation*}
u^n(0)=\ddot{z}_+^n(0)+\overline{\ddot{z}_-^n}(0),\qquad u^n(\tau)=\fe^{i\tau/\eps^2}\ddot{z}_+^n(\tau)+\fe^{-i\tau/\eps^2}\overline{\ddot{z}_-^n}(\tau).
\end{equation*}

\noindent{\sl Detailed numerical scheme}

\medskip
Following the same notations introduced in Subsection
\ref{subsec:LA}, choosing $y^0=y(0)=\phi_1$ and
$\dot{y}^0=\dot{y}(0)=\eps^{-2}\phi_2$, for $n=0,1,\ldots,$
$y^{n+1}$ and $\dot{y}^{n+1}$ are updated in the same way as
(\ref{IFSW})-(\ref{F_ini1}) except that
\begin{equation}\label{gen FA s}
\left\{
\begin{split}
&  z_\pm^{n+1}=\fe^{\frac{i\alpha}{2}\tau}z_\pm^{(0)}+\beta_1f_\pm\left(z_+^{(0)},z_-^{(0)}\right)+
\beta_2\dot{f}_\pm^{(0)},\\
& r^{n+1}=\frac{\sin(\omega\tau)}{\omega}\left(\dot{r}^{(0)} -\frac{\tau}{2}u^{(0)}\right)-\gamma_1f_r\left(z_+^{(0)},z_-^{(0)},r^{(0)};0\right),\\
&\dot{z}_\pm^{n+1}=\frac{i}{2}\left[\alpha z_\pm^{n+1}+f_\pm(z_+^{n+1},z_-^{n+1})\right],\\
&\dot{r}^{n+1}=\cos(\omega\tau)\left(\dot{r}^{(0)}-\frac{\tau}{2}u^{(0)}\right)-\frac{\tau}{2}u^{n+1}-\gamma_2f_r(z_+^{(0)},z_-^{(0)},r^{(0)};0)\\
&\ \qquad\quad -\gamma_3f_r(z_+^{n+1},z_-^{n+1},r^{n+1};\tau),\\
&\ddot{z}_\pm^{n+1}=\frac{i\alpha}{2}\dot{z}_\pm^{n+1}+\frac{i}{2}\frac{d}{ds}\left[f_\pm(z_+(s),z_-(s))\right]
\Big|_{\left\{z_\pm=z_\pm^{n+1},\ \dot{z}_\pm=\dot{z}_\pm^{n+1}\right\}},\\
\end{split}\right.
\end{equation}
with
\begin{equation}\label{gen FA e}
\left\{
\begin{split}
  &\dot{z}_\pm^{(0)}=\frac{i}{2}\left[\alpha z_\pm^{(0)}+f_\pm\left(z_+^{(0)},z_-^{(0)}\right)\right],\\
&u^{(0)}=\frac{i}{2}\left[\alpha\left(\dot{z}_+^{(0)}-\overline{\dot{z}_-^{(0)}}\right)+\dot{f}_+^{(0)}-\overline{\dot{f}_-^{(0)}}\right],\\
&\dot{f}_{\pm}^{(0)}=\frac{d}{ds}\left[f_\pm(z_+(s),z_-(s))\right]\Big|_{\left\{z_\pm=z_\pm^{(0)},\ \dot{z}_\pm=\dot{z}_\pm^{(0)}\right\}}.
 \end{split}\right.
\end{equation}

\begin{remark}
As it can be seen from the above integrators, one needs to evaluate functions $f_\pm^n$ and $\dot{f}_\pm^n$ in the scheme. In fact, these functions are given in the integral forms as (\ref{f_pm def}). In practice,
if explicit formulas are not available, they can be computed numerically via
the following composite trapezoidal rule
\begin{equation*}
f_\pm(z_+,z_-)\approx\frac{1}{N}\sum_{j=0}^{N-1}f\left(z_\pm+\fe^{i\theta_j}\overline{z_\mp}\right),
\qquad z_+, z_- \in \Bbb C,
\end{equation*}
where $N\in\bN$ is chosen to be large enough and $\theta_j=\frac{2\pi }{N}j$ for $j=0,1,\ldots,N$. Since the integrand  $f\left(z_\pm+\fe^{i\theta}\overline{z_\mp}\right)$ in (\ref{f_pm def}) is a periodic function with period $T=2\pi$, thus it is spectrally accurate to approximate the integrals in (\ref{f_pm def}) via the composite trapezoidal rule!
\end{remark}

\subsection{Another MTI based on MDF}
\label{SEC:2.5}
Based on the MDF (\ref{LSADz2}), we propose another MTI as follows.

For the first two equations in (\ref{pLSADz1}), the integrator follows (\ref{2VCFz})-(\ref{gen F appr}) totally.
As for the approximations to $r^n(\tau)$ and $\dot{r}^n(\tau)$, we follow the EWIs (\ref{gen r tau})-(\ref{f_pm appr}) by  setting $u^n=0$.

\medskip

\noindent{\sl Detailed numerical scheme}

\medskip

Following the same notations introduced in subsection \ref{subsec:LA}, choosing $y^0=y(0)=\phi_1$ and $\dot{y}^0=\dot{y}(0)=\eps^{-2}\phi_2$, for $n=0,1,\ldots,$ $y^{n+1}$ and $\dot{y}^{n+1}$ are updated in the same way as (\ref{IFSW}), (\ref{F s}) and (\ref{gen FA e}) except that
\begin{equation}\left\{
\begin{split} &r^{n+1}=\frac{\sin(\omega\tau)}{\omega}\dot{r}^{(0)}-\gamma_1f_r\left(z_+^{(0)},z_-^{(0)},r^{(0)};0\right),\\
&\dot{r}^{n+1}=\cos(\omega\tau)\dot{r}^{(0)}-\gamma_2f_r\left(z_+^{(0)},z_-^{(0)},r^{(0)};0\right)-
\gamma_3f_r\left(z_+^{n+1},z_-^{n+1},r^{n+1};\tau\right).
\end{split}\right.\label{gen F s}
\end{equation}
%with $z_{\pm}^{(0)},\ \dot{z}_\pm^{(0)},\ \dot{r}^{(0)},\ \dot{f}_\pm^{(0)}$  defined same as (\ref{gen FA e}).

\section{Classical numerical integrators}
\label{sec:5}

For comparison purpose, in this section, we recall two
classes of widely used numerical methods for directly
integrating the problem (\ref{WODE}).  The methods include
exponential wave integrators (EWIs) and conservative/nonconservative
finite difference (FD) integrators.
For simplicity of notations, we only consider the pure
power nonlinearity, i.e. $f$ in (\ref{WODE}) satisfies
(\ref{power}).

\subsection{Exponential wave integrators (EWIs)}
Similar to (\ref{r tau gr}) and (\ref{gen r tau}), we re-write the solution of (\ref{WODE}) near $t=t_n$
by using the variation-of-constant formula, i.e.
\begin{equation}\label{VCFy}
y(t_n+s)=\cos(\omega s)y(t_n)+\frac{\sin(\omega s)}{\omega}\dot{y}(t_n)-\int_0^s\frac{\sin(\omega(s-\theta))}{\eps^2\omega}f^n(\theta)d\theta,
\end{equation}where $f^n(\theta):=f(y(t_n+\theta))$.
Taking $s=\pm\tau$ in (\ref{VCFy}) and then summing them up, we have
\begin{equation}\label{VCFy tau}
y(t_{n+1})+y(t_{n-1})=2\cos(\omega \tau)y(t_n)-\int_0^\tau\frac{\sin(\omega(\tau-\theta))}
{\eps^2\omega}\left[f^n(\theta)+f^n(-\theta)\right]d\theta.
\end{equation}
Then EWIs approximate the integral term by proper quadratures.  For example, if a Gautschi's type quadrature  \cite{Gaustchi,Grimm,Dong,Lubich2} is applied, one can end up with the following EWI in Gautschi's type (EWI-G).  Following the same notations introduced in (\ref{IFSW}), the stabilized EWI-G \cite{Dong} reads
\begin{equation}\label{EWIG}
  y^{n+1}=\left\{
  \begin{split}
  \displaystyle
   &  -y^{n-1}+2\cos\left(\omega^n\tau\right)y^n
    -2G^n,
    & n\geq 1,\\
  \displaystyle
  &   \cos\left(\omega^0\tau\right)\phi_1+\frac{\sin\left(\omega^0\tau\right)}{\eps^2\omega^0}\phi_2-G^0,&n=0,
  \end{split}
  \right.
\end{equation}
where
\begin{align*}
&  G^n=\frac{1-\cos\left(\omega^n\tau\right)}{\eps^2(\omega^n)^2}
\left[g\left(|y^n|^2\right)y^n-\alpha^ny^n\right], \qquad n\ge0,\\ & \omega^n=\frac{\sqrt{1+\eps^2(\alpha+\alpha^{n})}}{\eps^2},\quad
 \alpha^{n}=\max\left\{\alpha^{n-1},g\left(|y^n|^2\right)\right\}, \mbox{~with~}\alpha^{-1}=0.
\end{align*}
Here a linear stabilizing term with stabilizing constant $\alpha^n$ is introduced so that the method
is unconditionally stable \cite[Theorem 6]{Dong}. Of course, one can use other ways to
filter oscillation in the resonance regime
\cite{Lubich1,Lubich2,Hairer1,Hairer2,Sanz} instead of the above linear stabilizing term. In addition,
if the approximation to $\dot{y}(t_n)$ is of interest, for example, evaluating the discrete energy, one can use
\begin{equation}\label{dy EWI}
  \dot{y}^{n+1}=\left\{
  \begin{split}
  &\dot{y}^{n-1}-2\omega\sin(\omega\tau)y^n-2\frac{\sin(\omega\tau)}{\eps^2\omega}g(|y^n|^2)y^n,\quad &n\geq1,\\
  &-\omega\sin(\omega\tau)y^0+\cos(\omega\tau)\dot{y}^0-\frac{\sin(\omega\tau)}{\eps^2\omega}g(|y^0|^2)y^0,\quad &n=0,
  \end{split}
  \right.
\end{equation}
which is derived similarly from the differentiation of (\ref{VCFy}) with respect to $s$ and then taking $s=\pm \tau$.

On the other hand, if the standard trapezoidal rule is applied
to approximate the integral in (\ref{VCFy tau}), then one can end up with the following EWI in Deuflhard's type (EWI-D) \cite{Deuflhard,Lubich1}.  Again, following the same notations introduced in (\ref{IFSW}),  EWI-D reads
\begin{equation}\label{EWID}
  y^{n+1}=\left\{
  \begin{split}
  \displaystyle
   &  -y^{n-1}+2\cos\left(\omega\tau\right)y^n
    - 2 D^n,
    & n\geq 1,\\
  \displaystyle
  &   \cos\left(\omega\tau\right)\phi_1+\frac{\sin\left(\omega\tau\right)}{\eps^2\omega}\phi_2
  -D^0,&n=0,
  \end{split}
  \right.
\end{equation}
where,
\begin{equation*}
  D^n = \frac{\tau\sin(\omega\tau)}{2\eps^2\omega}g\left(|y^n|^2\right)y^n,\quad n\geq 0.
\end{equation*}
Similarly, to approximate $\dot{y}(t_n)$, we can use the scheme (\ref{dy EWI}).

Generalizations of the above two EWIs based on (\ref{VCFy}) are the mollified impulse methods or EWIs with filters \cite{Garcia,Lubich1,Lubich2,Grimm}, which have been well-developed for solving problem (\ref{eq of MD}) with a uniform convergence and good energy preserving properties. Now with a stronger nonlinearity in the problem (\ref{WODE}), the scheme reads \begin{equation}\label{mollified impulse}
\left\{\begin{split}
y^{n+1}=&\cos(\omega\tau)y^n+\frac{\sin(\omega\tau)}{\omega}\dot{y}^n+\frac{\tau^2}{2\eps^2}\psi(\omega\tau)f\left(\phi(\omega\tau)y^n\right),\\
\dot{y}^{n+1}=&-\omega\sin(\omega\tau)y^n+\cos(\omega\tau)\dot{y}^n+\frac{\tau}{2\eps^2}\Big[\psi_0(\omega\tau)f\left(\phi(\omega\tau)y^n\right)\\
&+\psi_1(\omega\tau)f\left(\phi(\omega\tau)y^{n+1}\right)\Big],
\end{split}\right.
\end{equation}
where $\psi,\ \phi,\ \psi_0$ and $\psi_1$ are known as the filters under some consistent conditions \cite{Lubich1,Lubich2}. For example, two popular sets of filters mentioned in \cite{Lubich1,Lubich2} are choosing as
\begin{equation}\label{filter0}
\psi_0(\rho)=\cos(\rho)\psi_1(\rho),\qquad \psi_1(\rho)=\frac{\psi(\rho)}{{\rm sinc}(\rho)},
\end{equation}
with
\begin{equation}
\psi(\rho)=\phi(\rho){\rm sinc}(\rho),\qquad \phi(\rho)={\rm sinc}(\rho),\label{filter1}
\end{equation}
or
\begin{equation}
\psi(\rho)={\rm sinc}^2(\rho),\qquad \phi(\rho)=1,\label{filter2}
\end{equation}
where ${\rm sinc}(\rho)=\sin(\rho)/\rho$ for $\rho\in{\mathbb R}$. In the following, we refer to the EWIs (\ref{mollified impulse})-(\ref{filter0}) with filters (\ref{filter1}) as EWI-F1, and (\ref{mollified impulse})-(\ref{filter0}) with filters (\ref{filter2}) as EWI-F2.

For convergence results of the EWIs, we have the following theorem.
\begin{theorem}[Error bounds of EWIs]\label{tm3}
  For the EWI-G (\ref{EWIG}), EWI-D (\ref{EWID}), EWI-F1 (\ref{filter1}) and EWI-F2 (\ref{filter2}), under the assumption (\ref{assumption}), there exists a constant $\tau_0>0$ independent of $\eps$ and $n$, such that for any $0<\eps\leq 1$, when $0 <\tau\leq\tau_0$ satisfies $\tau\lesssim\eps^2$,
  \begin{equation}\label{erewi-g}
    |e^n|+\eps^2|\dot{e}^n|\lesssim\frac{\tau^2}{\eps^4},\qquad 0\leq n\leq\frac{T}{\tau}.
  \end{equation}
\end{theorem}
\begin{proof}
The proof proceeds in analogous lines as the method
used in \cite[Theorem 9]{Dong} towards the estimates in time or \cite{Grimm} and we omit the details here for brevity.\qed
\end{proof}

\subsection{Finite difference integrators}

For a sequence $\{y^n\}$, define the standard finite difference operators as
%\begin{align*}
\[ \delta_t^+ y^n:=\frac{y^{n+1}-y^n}{\tau},\quad
 \delta_t^- y^n:=\frac{y^{n}-y^{n-1}}{\tau},\quad
 \delta_t^2 y^n:=\frac{y^{n+1}-2y^n+y^{n-1}}{\tau^2}.
\]
%\end{align*}
Then a conservative Crank-Nicolson finite difference (CNFD)   integrator  for solving (\ref{WODE}) reads
\begin{equation}\label{CNFD}
 \eps^2\delta^2_ty^n+\left(\alpha+\frac{1}{\eps^2}\right)\frac{y^{n+1}
 +y^{n-1}}{2}+\hat{F}\left(y^{n+1},y^{n-1}\right)=0,\quad n=1,2,\ldots,
\end{equation}
where
\begin{equation*}
  \hat{F}\left(y^{n+1},y^{n-1}\right):=
  \frac{F\left(|y^{n+1}|^2\right)-F\left(|y^{n-1}|^2\right)}{|y^{n+1}|^2-|y^{n-1}|^2}
  \cdot\frac{y^{n+1}+y^{n-1}}{2}.
\end{equation*}
A semi-implicit finite difference (SIFD) integrator   reads
\begin{equation}\label{SIFD}
  \eps^2\delta^2_ty^n+\left(\alpha+\frac{1}{\eps^2}\right)\frac{y^{n+1}+y^{n-1}}
  {2}+g\left(|y^n|^2\right)y^n=0,\quad n=1,2,\ldots.
\end{equation}
An explicit finite difference (EXFD) integrator, which is known as
the famous St\"{o}rmer-Verlet or leap-frog method \cite{Lubich1,Lubich2,Reich}, reads
\begin{equation}\label{EXFD}
  \eps^2\delta^2_ty^n+\left(\alpha+\frac{1}{\eps^2}\right)y^n+g\left(|y^n|^2\right)y^n=0,\quad n=1,2,\ldots.
\end{equation}
Here the initial conditions are discretized as (\ref{EWID}), i.e.
\begin{equation} \label{inita123}
  y^0=\phi_1,\quad
  y^1=\cos\left(\omega\tau\right)\phi_1+\frac{\sin\left(\omega\tau\right)}{\eps^2\omega}\phi_2
  -\frac{\tau\sin(\omega\tau)}{2\eps^2\omega}g\left(|\phi_1|^2\right)\phi_1.
\end{equation}
In order that the methods CNFD and SIFD are stable
uniformly in the regime $0<\eps\ll 1$, here $y^1$ is computed according
to the EWI-D (\ref{EWID}) with $n=0$ instead of the classical way below.
In fact, if one adapts the usual way to obtain $y^1$
as
\begin{equation} \label{inita124}
y^1=\phi_1+\frac{\tau\phi_2}{\eps^2}-\frac{\tau^2}{2\eps^2}
\left[\left(\alpha+\frac{1}{\eps^2}\right)\phi_1+g\left(|\phi_1|^2\right)\phi_1\right].
\end{equation}
Our numerical results suggest that it would cause severe instability issue
when $\tau=O(1)$ and $0<\eps\ll 1$. Thus we adopt (\ref{inita123}) instead of (\ref{inita124})
to discretize the initial data since we want to consider $0<\eps\le 1$, especially $0<\eps\ll1$.

For the above CNFD, SIFD and EXFD integrators, all are time
symmetric. CNFD is implicit, SIFD is implicit but can be
solved very efficiently, and EXFD is explicit. For CNFD, it
 conserves the following energy in the discretized level, i.e.
\begin{eqnarray*}
  E^n&:=&\eps^2\left|\delta^+_ty^n\right|^2+\left(\alpha+\frac{1}{\eps^2}\right)\frac{|y^{n+1}|^2+|y^n|^2}{2}
  +\frac{F\left(|y^{n+1}|^2\right)+F\left(|y^n|^2\right)}{2}\\
  &\equiv&E^0, \qquad n=0,1,\ldots
\end{eqnarray*}
However, at each step, a fully nonlinear equation needs to be solved,
which might be quite time-consuming. In fact, if the nonlinear
equation is not solved very accurately, then the above quantity
will not be conserved in practical computation \cite{Bao5}.
Thus CNFD is usually not adopted in practical computation,
especially for partial differential equations in high dimensions.
EXFD is very popular and powerful when $\eps=O(1)$,
however, it suffers from a server stability constraint
$\tau\lesssim\eps^2$  when $0<\eps\ll1$ \cite{Dong}.

For the above finite difference integrators, defining the error functions again as (\ref{error_fun}), we have the following convergence results, providing the exact solution $y(t)$ to (\ref{WODE}) satisfying
\begin{equation}\label{assumption2}
y(t)\in C^4(0,T),\qquad
\left\|\frac{d^m}{dt^m}y(t)\right\|_{L^\infty(0,T)}\lesssim\frac{1}{\eps^{2m}},\quad m=0,1,2,3,4.
\end{equation}
\begin{theorem}[Error bounds of CNFD and SIFD]\label{tm4}
For the CNFD (\ref{CNFD}) and SIFD (\ref{SIFD}),
under the assumption (\ref{assumption2}),
there exists a constant $\tau_0>0$ independent of $\eps$ and $n$,
such that for any $0<\eps\leq 1$
  \begin{equation}\label{ercnfd}
    |e^n|+\eps^2|\dot{e}^n|\lesssim\frac{\tau^2}{\eps^6},\qquad 0\leq n\leq\frac{T}{\tau},
    \qquad 0 <\tau\leq\tau_0.
  \end{equation}
\end{theorem}

\begin{proof}
  The proof  proceeds in analogous lines as the technique
  used in \cite[Theorem 2 and 5]{Dong}, and we omit the details here for brevity.\qed
\end{proof}

\begin{theorem}[Error bound of EXFD]\label{tm5}
For the EXFD (\ref{EXFD}), under the assumption
(\ref{assumption2}), there exists a constant $\tau_0>0$
independent of $\eps$ and $n$, such that for any $0<\eps\leq 1$,
when $0 <\tau\leq\tau_0$ satisfying $\tau\lesssim\eps^2$,
  \begin{equation} \label{erexfd}
    |e^n|+\eps^2|\dot{e}^n|\lesssim\frac{\tau^2}{\eps^6},\qquad 0\leq n\leq\frac{T}{\tau}.
  \end{equation}
\end{theorem}

\begin{proof}
  The proof  proceeds in analogous lines as the technique used in
  \cite[Theorem 3]{Dong} and  the details are omitted here for brevity.\qed
\end{proof}

\section{Numerical comparison results}
\label{sec:6}

In this section, we present numerical comparison results between
the proposed MTIs including MTI-FA and MTI-F, EWIs including EWI-G, EWI-D, EWI-F1 and EWI-F2,
and classical finite difference integrators including CNFD, SIFD and EXFD.
We will compare their accuracy for fixed $\eps=O(1)$ and their meshing strategy
(or $\eps$-resolution) in the parameter regime when $0<\eps\ll 1$.
To quantify the convergence, we introduce two error functions:
\begin{equation}
  \fe^{\eps,\tau}(T):=\left|y(T)-y^{M}\right|,\quad
  \fe^{\tau}_{\infty}(T):=\max_{\eps}\left\{e^{\eps,\tau}(T)\right\},
\end{equation}
where $T=t_M$ with $t_M=M\tau$.

\subsection{Results for power nonlinearity}
%Convergence study for $0<\eps\ll 1$}

The nonlinearity in the  problem (\ref{WODE}) is taken as the pure power nonlinearity (\ref{power})
with coefficients and initial conditions chosen as
\begin{equation}\label{test_eg}
\alpha=2,\quad g\left(|y|^2\right)=|y|^2,\quad \phi_1=1,\quad \phi_2=1.
\end{equation}
Since the analytical solution to this problem is not available, the `exact' solution is
obtained numerically by the  MTI-FA (\ref{IFSW}) with (\ref{FS_r}) under a very small time step $\tau=10^{-6}$.

Table \ref{tbFSWD} lists the errors of the method MTI-FA (\ref{IFSW}) with (\ref{FS_r}) under different
$\eps$ and $\tau$, and Table \ref{tbLSAD} shows similar results
for the method MTI-F (\ref{IFSW}) with (\ref{F s}).
For comparison, Table \ref{tbEWI} shows the errors of EWI-G (\ref{EWIG}) and EWI-D (\ref{EWID}),
Table \ref{tbMI} shows the errors of EWI-F1 (\ref{filter1}) and EWI-F2 (\ref{filter2}),
and Table \ref{tbFD} lists  the errors of CNFD (\ref{CNFD}), SIFD (\ref{SIFD}) and EXFD (\ref{EXFD}).

\begin{table}[t!]\small
\tabcolsep 0pt
\caption{Error analysis of MTI-FA: $\fe^{\eps,\tau}(T)$ and $\fe^{\tau}_{\infty}(T)$ with $T=4$ and convergence rate. Here and after, the convergence rate is obtained by $\frac{1}{2}\log_2\left(\frac{\fe^{\eps,4\tau}(T)}{\fe^{\eps,\tau}(T)}\right)$.}\label{tbFSWD}
\begin{center}
\begin{tabular*}{1\textwidth}{@{\extracolsep{\fill}}llllllll}
\hline
$\fe^{\eps,\tau}(T)$        & $\tau_0=0.2$	&$\tau_0/2^2$	&$\tau_0/2^4$	&$\tau_0/2^6$	 &$\tau_0/2^8$& $\tau_0/2^{10}$&	 $\tau_0/2^{12}$\\
\hline
$\eps_0=0.5$	&5.71E\,--1	&5.28E\,--2	&3.40E\,--3	&2.14E\,--4	&1.34E\,--5	&8.36E\,--7	 &5.21E\,--8\\
rate &---       &1.72	    &1.98	    &2.00	    &2.00	    &2.00	    &2.00\\ \hline
$\eps_0/2^1$    &3.14E\,--1	&5.56E\,--2	&5.70E\,--3	&3.51E\,--4	&2.17E\,--5	&1.35E\,--6	 &8.43E\,--8\\
		  rate  &---        &1.25       &1.64	    &2.01	    &2.01	    &2.00	     &2.00\\ \hline
$\eps_0/2^2$    &1.59E\,--1	&1.53E\,--1	&4.58E\,--2	&2.80E\,--3	&1.56E\,--4	&9.36E\,--6	 &5.79E\,--7\\
		   rate &---        &0.03       &0.87       &2.02	    &2.08	    &2.03	     &2.01\\ \hline
$\eps_0/2^3$    &5.90E\,--3	&1.59E\,--2	&1.25E\,--2	&5.90E\,--3	&2.51E\,--4	&1.16E\,--5	 &6.58E\,--7\\
		   rate &---        &-0.72      &0.17       &0.54       &2.28	    &2.22	     &2.07\\ \hline
$\eps_0/2^4$    &6.70E\,--3	&5.40E\,--3	&8.60E\,--3	&7.30E\,--3	&2.60E\,--3	&1.33E\,--4	 &6.82E\,--6\\
		   rate &---        &0.16       &0.34       &0.12       &0.74       &2.14	     &2.14\\ \hline
$\eps_0/2^5$    &1.10E\,--3	&1.00E\,--3	&6.36E\,--4	&1.30E\,--3	&1.30E\,--3	&2.77E\,--4	 &2.06E\,--5\\
           rate &---	    &0.07		&0.33       &-0.52      &0.00       &1.12        &1.87\\ \hline
$\eps_0/2^6$    &5.96E\,--4	&2.18E\,--5	&5.96E\,--4	&4.10E\,--4	&5.97E\,--4	&5.18E\,--4	 &1.78E\,--4\\
           rate &---		&2.39		&-2.39      &0.27       &-0.27      &0.10        &0.77\\ \hline
$\eps_0/2^8$    &6.51E\,--6	&7.14E\,--6	&1.04E\,--5	&7.48E\,--6	&7.00E\,--6	&3.48E\,--6	 &1.03E\,--5\\	
      rate      &---		&-0.07		&-0.27      &0.24       &0.05       &0.50        &0.78\\	 \hline	
$\eps_0/2^{10}$ &2.32E\,--7	&4.85E\,--7	&2.66E\,--7	&2.79E\,--6	&2.52E\,--6	&5.01E\,--8	 &2.66E\,--6\\
		   rate &---	    &-0.53		&0.43       &-1.70      &0.07       &2.83        &-2.87\\ \hline	
$\eps_0/2^{12}$ &9.87E\,--8	&4.34E\,--8	&6.68E\,--8	&2.33E\,--8	&7.56E\,--8	&1.19E\,--7	 &1.12E\,--7\\	
           rate &---		&0.59		&-0.31      &0.76       &-0.85      &-0.33       &0.04\\	 \hline
$\eps_0/2^{14}$ &3.38E\,--8	&3.77E\,--8	&3.84E\,--8	&3.55E\,--8	&3.49E\,--8	&3.45E\,--8	 &3.43E\,--8\\	
           rate &---		&-0.08		&-0.01      &0.06       &0.01      &0.01      &0.01\\	
\hline \hline
$\fe^{\tau}_{\infty}(T)$ &5.71E\,--1	&1.53E\,--1	&4.58E\,--2	&7.30E\,--3	&2.60E\,--3	 &5.18E\,--4	 &1.78E\,--4\\
             rate &---          &0.95	    &0.87	    &1.32	    &0.74	    &1.16	    &0.77\\
\hline
\end{tabular*}
\end{center}
\end{table}

\begin{table}[t!]
\tabcolsep 0pt
\caption{Error analysis of MTI-F: $\fe^{\eps,\tau}(T)$ and $\fe^{\tau}_{\infty}(T)$ with $T=4$ and convergence rate. }\label{tbLSAD}
\begin{center}
\begin{tabular*}{1\textwidth}{@{\extracolsep{\fill}}llllllll}
\hline
$\fe^{\eps,\tau}(T)$        & $\tau_0=0.2$	&$\tau_0/2^2$	&$\tau_0/2^4$	&$\tau_0/2^6$	 &$\tau_0/2^8$& $\tau_0/2^{10}$&	 $\tau_0/2^{12}$\\
\hline
$\eps_0=0.5$	&5.33E\,--1	&4.05E\,--2	&2.80E\,--3	&1.84E\,--4	&1.16E\,--5	&7.27E\,--7	 &4.53E\,--8\\
	       rate &---        &1.86	    &1.93	    &1.96	    &1.99	    &2.00	     &2.00\\ \hline
$\eps_0/2$      &3.71E\,--1	&5.54E\,--2	&5.60E\,--3	&3.48E\,--4	&2.16E\,--5	&1.34E\,--6	 &8.38E\,--8\\
	       rate &---        &1.37	    &1.65	    &2.00	    &2.01	    &2.00	     &2.00\\ \hline
$\eps_0/2^2$	&2.78E\,--1	&1.60E\,--1	&4.51E\,--2	&2.80E\,--3	&1.55E\,--4	&9.35E\,--6	 &5.79E\,--7\\
	       rate &---        &0.40	    &0.91	    &2.00	    &2.09	    &2.03	     &2.01\\ \hline
$\eps_0/2^3$    &4.95E\,--2	&1.68E\,--2	&1.20E\,--2	&5.80E\,--3	&2.50E\,--4	&1.16E\,--5	 &6.57E\,--7\\
	       rate &---        &0.78	    &0.24	    &0.52	    &2.27	    &2.22	     &2.07\\ \hline
$\eps_0/2^4$	&1.07E\,--1	&9.20E\,--3	&8.70E\,--3	&7.30E\,--3	&2.60E\,--3	&1.33E\,--4	 &6.82E\,--6\\
	      rate  &---        &1.77	    &0.04	    &0.13	    &0.87	    &2.14	     &2.14\\ \hline
$\eps_0/2^5$	&6.15E\,--2	&3.90E\,--3	&8.00E\,--4	&1.40E\,--3	&1.30E\,--3	&2.76E\,--4	 &2.06E\,--5\\
	      rate  &---        &1.99	    &1.14	    &-0.40	    &0.05	    &1.12	     &1.87\\ \hline
$\eps_0/2^6$	&1.14E\,--1	&4.80E\,--3	&8.54E\,--4	&4.24E\,--4	&5.97E\,--4	&5.18E\,--4	 &1.78E\,--4\\
	       rate &---        &2.28	    &1.25	    &0.50	    &-0.25	    &0.10	     &0.77\\ \hline
$\eps_0/2^8$	&2.60E\,--2	&1.40E\,--3	&9.98E\,--5	&1.31E\,--5	&7.36E\,--6	&3.50E\,--6	 &1.03E\,--5\\
	       rate &---        &2.11	    &1.91	    &1.47	    &0.41	    &0.54	     &-0.78\\ \hline
$\eps_0/2^{10}$ &1.23E\,--1	&5.30E\,--3	&2.91E\,--4	&2.04E\,--5	&3.61E\,--6	&1.20E\,--7	 &2.67E\,--6\\
	       rate &---        &2.27	    &2.09	    &1.92	    &1.25	    &2.45	     &-2.24\\ \hline
$\eps_0/2^{12}$ &1.35E\,--1	&6.00E\,--3	&3.41E\,--4	&2.08E\,--5	&1.25E\,--6	&2.36E\,--7	 &1.53E\,--7\\
 	       rate &---        &2.24	    &2.07	    &2.02	    &2.03	    &1.20	     &0.31\\ \hline
$\eps_0/2^{14}$ &4.57E\,--2	&2.30E\,--3	&1.36E\,--4	&8.28E\,--6	&3.27E\,--7	&1.67E\,--7	 &1.97E\,--7\\
           rate &---	    &2.15		&2.04       &2.02       &2.32       &0.49        &-0.12\\	
\hline \hline
$\fe^{\tau}_{\infty}(T)$ &5.33E\,--1	&1.60E\,--1	&4.51E\,--2	&7.30E\,--3	&2.60E\,--3	 &5.18E\,--4 &1.78E\,--4\\
	                rate &---           &0.87	    &0.91	    &1.31	    &0.74	    &1.16	    &0.77\\
\hline
\end{tabular*}
\end{center}
\end{table}

\begin{table}[t!]
\centering
\tabcolsep 0pt
\caption{Error analysis of EWI-G and EWI-D:  $\fe^{\eps,\tau}(T)$ with $T=4$ and convergence rate.}\label{tbEWI}
\begin{tabular*}{1\textwidth}{@{\extracolsep{\fill}}lllllll}
\hline
EWI-G	    &$\tau_0=0.2$ &$\tau_0/2^2$ &$\tau_0/2^4$ &$\tau_0/2^6$ &$\tau_0/2^8$ &$\tau_0/2^{10}$\\
 \hline
$\eps_0=0.5$	 &1.09E\,--2	&1.59E\,--3	&1.01E\,--4	&6.36E\,--6	&3.97E\,--7	&2.44E\,--8\\
	        rate &---           &1.39	    &1.98	    &2.00	    &2.00	    &2.01\\ \hline
$\eps_0/2$       &2.34E+0	    &2.74E\,--2	&1.75E\,--3	&1.10E\,--4	&6.86E\,--6	&4.29E\,--7\\
		    rate &---           &3.21       &1.98	    &2.00	    &2.00	    &2.00\\ \hline
$\eps_0/2^2$     &9.65E\,--1    &9.87E\,--1	&6.50E\,--2	&3.90E\,--3	&2.43E\,--4	&1.52E\,--5\\
			rate &---           &-0.02      &1.96       &2.03	    &2.00	    &2.00\\ \hline
$\eps_0/2^3$     &3.06E\,--1	&1.90E\,--1	&2.68E+0	&2.20E\,--2	&1.18E\,--3	&7.33E\,--5\\
		    rate &---           &0.34       &1.91       &3.46       &2.11	    &2.01\\ \hline
$\eps_0/2^4$     &2.73E\,--1	&3.01E\,--1	&3.05E\,--1	&2.41E+0	&5.40E\,--2	&3.08E\,--3\\
			rate &---           &-0.07      &0.01       &-1.49      &2.74	    &2.07\\ \hline
$\eps_0/2^6$     &2.03E+0	    &2.06E+0	&1.95E+0	&2.09E+0	&2.09E+0	&3.56E\,--1\\
			rate &---           &-0.01      & 0.04      &-0.05      &0.00	    &1.28\\ \hline
$\eps_0/2^8$     &2.66E+0	    &2.66E+0	&2.68E+0	&2.65E+0	&2.71E+0	&2.63E+0\\
		    rate &---           &0.00       &-0.01      &0.01       &-0.02	    &0.02\\
\hline \hline
EWI-D        &$\tau_0=0.2$ &$\tau_0/2^2$ &$\tau_0/2^4$ &$\tau_0/2^6$ &$\tau_0/2^8$ &$\tau_0/2^{10}$\\
\hline
$\eps_0=0.5$     &1.02E\,--1	&5.97E\,--3	&3.66E\,--4	&2.29E\,--5	&1.43E\,--6	&9.05E\,--8\\
	        rate &---           &2.05	    &2.01	    &2.00	    &2.00	    &1.99\\ \hline
$\eps_0/2$       &7.61E\,--2	&3.25E\,--2	&1.52E\,--3	&9.37E\,--5	&5.85E\,--6	&3.66E\,--7\\
		    rate &---           &0.61       &2.21	    &2.01	    &2.00	    &2.00\\ \hline
$\eps_0/2^2$     &5.66E\,--1	&6.04E\,--1	&2.19E\,--2	&1.19E\,--3	&7.36E\,--5	&4.60E\,--6\\
			rate &---           &-0.05      &2.39       &2.10	    &2.01	    &2.00\\ \hline
$\eps_0/2^3$     &1.10E\,--1	&2.83E\,--1	&2.96E\,--1	&2.56E\,--3	&1.41E\,--4	&8.76E\,--6\\
		    rate &---           &0.68       &-0.03      &3.43       &2.09	    &2.00\\ \hline
$\eps_0/2^4$     &3.78E\,--1	&5.85E\,--2	&1.52E\,--1	&1.57E\,--1	&1.16E\,--3	&6.47E\,--5\\
		    rate &---           &1.35       &0.69       &-0.02      & 3.54      &2.08\\ \hline
$\eps_0/2^6$     &1.03E+0	    &2.09E\,--1	&5.92E\,--2	&5.74E\,--3	&1.17E\,--2	&1.20E\,--2\\
			rate &---           &1.15       &0.91       &1.68       &-1.17	    &-0.02\\ \hline
$\eps_0/2^8$     &1.39E\,--1	&1.32E\,--2	&7.17E\,--3	&1.92E\,--3	&6.57E\,--4	&6.80E\,--5\\
			rate &---           &1.70       &0.44       &0.95       &0.77	    &1.64\\
\hline
\end{tabular*}
\end{table}

\begin{table}[t!]
\centering
\tabcolsep 0pt
\caption{Error analysis of EWI-F1 and EWI-F2:  $\fe^{\eps,\tau}(T)$ with $T=4$ and convergence rate.}\label{tbMI}
\begin{tabular*}{1\textwidth}{@{\extracolsep{\fill}}lllllll}
\hline
MI-F1	         &$\tau_0=0.2$  &$\tau_0/2^2$ &$\tau_0/2^4$ &$\tau_0/2^6$ &$\tau_0/2^8$ &$\tau_0/2^{10}$\\
 \hline
$\eps_0=0.5$	 &9.73E\,--1	&6.98E\,--2	  &4.40E\,--3	&2.72E\,--4	  &1.70E\,--5	&1.01E\,--6\\
	rate         &---           &1.90	      &2.00	        &2.00	      &2.00	        &2.04\\ \hline
$\eps_0/2$       &1.70E+0	    &1.30E\,--1	  &4.87E\,--2	&3.20E\,--3	  &2.03E\,--4	&1.26E\,--5\\
	rate         &---           &1.85	      &0.71	        &1.96	      &1.99	        &2.00\\ \hline
$\eps_0/2^2$     &3.49E\,--1	&3.49E\,--1	  &9.81E\,--1	&1.01E\,--1	  &6.40E\,--3	&4.02E\,--4\\
	rate         &---           &0.00	      &-0.74	    &1.64	      &1.99	        &2.00\\ \hline
$\eps_0/2^3$     &2.76E+0	    &2.76E+0	  &2.76E+0	    &1.01E+0	  &3.33E\,--2	&1.90E\,--3\\
	rate         &---           &0.00	      &0.00	        &0.73	      &2.46	        &2.08\\ \hline
$\eps_0/2^4$     &2.26E+0	    &2.26E+0	  &2.26E+0	    &2.26E+0	  &1.35E+0	    &7.63E\,--2\\
	rate         &---           &0.00	      &0.00	        &0.00	      &0.37	        &2.07\\ \hline
$\eps_0/2^6$     &2.04E+0	    &2.04E+0	  &2.04E+0	    &2.04E+0	  &2.04E+0	    &2.04E+0\\
rate             &---           &0.00	      &0.00	        &0.00         &0.00	        &0.00\\ \hline	
$\eps_0/2^8$     &2.66E+0	    &2.66E+0	  &2.66E+0	    &2.66E+0	  &2.66E+0	    &2.66E+0\\
rate             &---           &0.00	      &0.00	        &0.00         &0.00	        &0.00\\
\hline \hline
MI-F2            &$\tau_0=0.2$  &$\tau_0/2^2$ &$\tau_0/2^4$ &$\tau_0/2^6$ &$\tau_0/2^8$  &$\tau_0/2^{10}$\\
\hline
$\eps_0=0.5$     &2.18E\,--1	&1.30E\,--2	  &8.15E\,--4	&5.09E\,--5	  &3.13E\,--6	 &1.44E\,--7\\
	     rate    &---           &2.03	      &2.00	        &2.00	      &2.01 	     &2.21\\ \hline
$\eps_0/2$       &2.00E+0	    &1.54E\,--1	  &1.17E\,--2	&7.41E\,--4	  &4.63E\,--5	 &2.81E\,--6\\
	rate         &---           &1.85	      &1.86	        &1.99	      &2.00	         &2.02\\ \hline
$\eps_0/2^2$     &2.12E\,--1	&4.99E\,--1	  &3.68E\,--1	&2.48E\,--2	  &1.60E\,--3	 &9.66E\,--5\\
	rate         &---           &-0.62	      &0.22 	    &1.95	      &2.00	         &2.02\\ \hline
$\eps_0/2^3$     &2.77E+0	    &2.77E+0	  &2.75E+0	    &1.74E\,--1	  &7.50E\,--3	 &4.55E\,--4\\
	rate         &---           &0.00	      &0.00	        &1.99	      &2.26	         &2.02\\ \hline
$\eps_0/2^4$     &2.25E+0	    &2.30E+0	  &2.30E+0	    &2.21E+0	  &3.32E\,--1	 &1.86E\,--2\\
	rate         &---           &-0.01	      &0.00	        &0.03	      &1.37	         &2.08\\ \hline
$\eps_0/2^6$     &2.04E+0	    &2.04E+0	  &2.03E+0	    &2.08E+0	  &2.09E+0	     &1.99E+0\\
rate             &---           &0.00	      &0.00         &-0.01	      &0.00	         &0.03\\ \hline
$\eps_0/2^8$     &2.66E+0	    &2.66E+0	  &2.66E+0	    &2.66E+0	  &2.67E+0	     &2.63E+0\\
rate             &---           &0.00	      &0.00         &0.00         &0.00          &0.01\\
\hline
\end{tabular*}
\end{table}

\begin{table}[t!]
\centering
\tabcolsep 0pt
\caption{Error analysis of CNFD and SIFD:  $\fe^{\eps,\tau}(T)$ with $T=4$ and convergence rate.}\label{tbFD}
\begin{tabular*}{1\textwidth}{@{\extracolsep{\fill}}lllllll}
\hline
CNFD        &$\tau_0=0.2$   &$\tau_0/2^2$  &$\tau_0/2^4$  &$\tau_0/2^6$ &$\tau_0/2^8$ &$\tau_0/2^{10}$\\
 \hline
$\eps_0=0.5$     &3.24E\,--1	  &4.49E\,--1	   &2.75E\,--2	&1.71E\,--3	&1.07E\,--4	&6.69E\,--6\\
	        rate &---             &-0.24	       &2.01	    &2.00	    &2.00	    &2.00\\ \hline
$\eps_0/2$       &1.75E+0	      &2.42E+0	       &1.90E\,--1	&3.41E\,--2	&2.21E\,--3	&1.38E\,--4\\
			rate &---             &-0.23           &1.84        &1.24	    &1.97	    &2.00\\ \hline
$\eps_0/2^2$     &1.05E+0	      &1.50E+0	       &5.02E\,--1	&3.54E\,--1	&1.94E\,--1	&1.24E\,--2\\
			rate &---	          &-0.26           &0.79        &0.25       &0.43       &1.98\\ \hline
$\eps_0/2^3$     &3.78E\,--1	  &1.78E+0	       &3.71E\,--1  &2.69E+0	&2.60E+0	&3.93E\,--1\\
            rate &---             &1.11            &1.13        &1.43       &0.02       &1.36\\ \hline
$\eps_0/2^4$     &6.49E\,--2	  &1.51E\,--1	   &1.05E+0  	&7.87E\,--1	&5.36E\,--2	&2.48E+0\\
            rate &---             &0.61            &-1.40       &0.21       &1.94       &-2.76\\ \hline
$\eps_0/2^6$     &1.95E+0	      &1.95E+0	       &1.97E+0	    &3.55E\,--1	&2.46E+0	&1.25E+0\\
            rate &---             &0.00            &-0.01       &1.24       &-1.40      &0.49\\ \hline
$\eps_0/2^8$     &3.63E\,--1	  &3.64E\,--1	   &3.64E\,--1	&3.63E\,--1	&5.75E\,--2	&2.49E+0\\
            rate &---             &0.00            &0.00        &0.00       &1.33       &-2.72\\  \hline \hline
SIFD        &$\tau_0=0.2$   &$\tau_0/2^2$  &$\tau_0/2^4$  &$\tau_0/2^6$ &$\tau_0/2^8$ &$\tau_0/2^{10}$\\ \hline
$\eps_0=0.5$	 &7.61E\,--1	  &2.88E\,--1	   &1.76E\,--2	&1.09E\,--3	&6.83E\,--5	 &4.27E\,--6\\
	        rate &---             &0.70	           &2.02	    &2.00	    &2.00	    &2.00\\ \hline
$\eps_0/2$       &2.32E\,--1	  &1.25E+0	       &2.13E\,--1	&2.82E\,--2	&1.82E\,--3	&1.14E\,--4\\
			rate &---             &-1.21           &1.28        &1.46	    &1.98	    &2.00\\ \hline
$\eps_0/2^2$     &1.61E+0	      &1.15E+0	       &1.73E+0	    &5.08E\,--1	&1.83E\,--1	&1.17E\,--2\\
		    rate &---             &0.24            &-0.29       &0.88       &0.74       &1.98\\ \hline
$\eps_0/2^3$     &2.42E\,--1	  &6.85E\,--1	   &5.05E\,--1	&2.21E+0	&2.50E+0	&3.85E\,--1\\
            rate &---             &-0.75           &0.22        &-1.06      &-0.09      &1.35\\ \hline
$\eps_0/2^4$     &1.13E\,--1	  &4.44E\,--2	   &1.91E+0   	&3.28E\,--1	&1.58E+0	&2.48E+0\\
            rate &---             &0.67            &-2.71       &1.27       &-1.13      &-0.33\\ \hline
$\eps_0/2^6$     &1.95E+0	      &1.95E+0	       &1.92E+0  	&6.89E\,--1	&2.05E+0	&6.26E\,--1\\
            rate &---             &0.00            &0.01        &0.74       &-0.78      &0.86\\ \hline
$\eps_0/2^8$     &3.63E\,--1	  &3.63E\,--1	   &3.65E\,--1	&3.63E\,--1	&9.42E\,--2	&2.70E+0\\
            rate &---             &0.00            &-0.01       &0.01       &0.97       &-2.42\\
\hline
\end{tabular*}
\end{table}

\begin{table}[t!]
\centering
\tabcolsep 0pt
\caption{Error analysis of EXFD:  $\fe^{\eps,\tau}(T)$ with $T=4$ and convergence rate.}\label{tbEXFD}
\begin{tabular*}{1\textwidth}{@{\extracolsep{\fill}}lllllll}
\hline
EXFD            &$\tau_0=0.2$   &$\tau_0/2^2$  &$\tau_0/2^4$  &$\tau_0/2^6$ &$\tau_0/2^8$ &$\tau_0/2^{10}$\\ \hline
$\eps_0=0.5$	&8.84E\,--1	    &7.52E\,--2	   &4.66E\,--3	  &2.90E\,--4	    &1.81E\,--5	  &1.13E\,--6\\
	       rate &---            &1.78	       &2.01	      &2.00	            &2.00	      &2.00\\ \hline
$\eps_0/2$      &unstable	           &2.51E+0	   &1.15E\,--1	  &6.49E\,--3	    &4.03E\,--4	  &2.51E\,--5\\
	      rate  &---            &---	           &2.22	      &2.07	            &2.01	      &2.00\\ \hline
$\eps_0/2^2$    &unstable	 &unstable	  &1.76E+0	      &6.36E\,--1	    &3.87E\,--2	  &2.41E\,--3\\
           rate &---	        &---       &---	          &0.73  	        &2.02	      &2.00\\ \hline
$\eps_0/2^3$    &unstable	&unstable	 &unstable	    &1.34E+0	        &1.23E+0	  &3.25E\,--2\\
	       rate &---           &---	           &---	          &---	            &0.06	      &2.62\\ \hline
$\eps_0/2^4$    &unstable  &unstable   &unstable   &unstable	 &9.96E\,--1	  &3.37E\,--1\\
           rate &---  &---	 &---	 &---	        &---	          &0.78\\ \hline
$\eps_0/2^6$    &unstable	  &unstable	   &unstable	&unstable	&unstable	&unstable \\
           rate &---	 &---	 &---	 &---	     &---	          &---\\ \hline
$\eps_0/2^8$    &unstable	&unstable	   &unstable	  &unstable	    &unstable	 &unstable\\
           rate &---	  &---	    &---	  &---	         &---	          &---\\
\hline
\end{tabular*}
\end{table}\normalsize

Based on Tables \ref{tbFSWD}-\ref{tbEXFD} and additional results not shown here for brevity, the following observations can be drawn:

\bigskip

1). For any fixed $\eps$ under $0<\eps\le 1$, when $\tau$ is sufficiently small, e.g. $\tau \lesssim \eps^2$,
all the numerical methods are  second-order accurate (cf. each row in the upper triangle of the  tables).
When $\eps=O(1)$, e.g. $\eps=0.5$, the errors are in the same magnitude for
all the numerical methods under fixed $\tau$ (cf. first row in the tables);
on the contrary, when $\eps$ is small, under fixed $\tau$ small enough, the errors in MTI-FA and MTI-F are
several order smaller in magnitude than those in EWI-G, EWI-D, EWI-F1 and EWI-F2, and the errors in EWI-G, EWI-D, EWI-F1 and EWI-F2 are
a few order smaller in magnitude than those in CNFD, SIFD and EXFD (cf. right bottom parts in the upper triangle of the  tables).

2). Both MTI-FA and MTI-F are uniformly accurate for $0<\eps\le 1$ and converge linearly in
$\tau$ (cf. last row in Tables \ref{tbFSWD}\&\ref{tbLSAD}). In addition, for fixed $\tau$, when $0<\eps\ll1$, MTI-FA converges
quadratically in term of $\eps$ (cf. each column in the lower triangle of Table \ref{tbFSWD}); MTI-F is uniformly bounded
(cf. each column in the lower triangle of Table \ref{tbLSAD}). These results confirm our analytical results in (\ref{FSWDbound}), (\ref{FSWDerror2}), (\ref{LSADbound}) and (\ref{LSADerror2}).
EWI-G, EWI-D, EWI-F1, EWI-F2, CNFD, SIFD and EXFD are not uniformly accurate for $0<\eps\le 1$ (cf. each column in Tables \ref{tbEWI}\&\ref{tbMI}). In fact, for fixed $\tau$ small enough, when $\eps$ decreases, the errors for EWI-G, EWI-D, EWI-F1 and EWI-F2 increase in term of $\eps^{-4}$ (cf. last row in Table \ref{tbEWI}),
and resp., for CNFD, SIFD and EXFD in term of $\eps^{-6}$ (cf. last row in Table \ref{tbMI}).
These results confirm our analytical results in (\ref{erewi-g}), (\ref{ercnfd}) and (\ref{erexfd}).

3). In summary, when $\eps=O(1)$, all the methods perform the same in term of accuracy, however,
EXFD is the simplest and cheapest one in term of computational cost. On the contrary, when $0<\eps<1$, especially $0<\eps\ll1$, both MTI-FA and MTI-F perform much better than the other classical methods. In fact, in order
to compute `correct' numerical solution, in the regime of $0<\eps\ll1$, the $\eps$-scalability (or meshing stragety) for MTI-FA and MTI-F is: $\tau=O(1)$ which is independent of $\eps$, where EWI-G, EWI-D, EWI-F1 and EWI-F2 need to choose $\tau=O(\eps^2)$ and CNFD, SIFD and EXFD need to choose $\tau=O(\eps^3)$.

\subsection{Results of MTIs for general gauge invariant nonlinearity}

The nonlinearity and initial conditions in the  problem (\ref{WODE}) are  chosen as
\begin{equation*}
\alpha=3,\qquad f(y)=\sin^2(|y|^2)y,\qquad \phi_1=1,\qquad \phi_2=1.
\end{equation*}
Again, the `exact' solution is
obtained numerically by the  MTI-FA  (\ref{IFSW}) with (\ref{gen FA s}) and (\ref{gen FA e})
 under a very small time step.

Table {\ref{MTI FA gen}} shows the errors of the method MTI-FA
(\ref{IFSW}) with (\ref{gen FA s}) and (\ref{gen FA e}) under different
$\eps$ and $\tau$, and Table \ref{MTI F gen} lists similar results for the method
MTI-F (\ref{IFSW}) with (\ref{gen F s}). The results for EWI-G, EWI-D, EWI-F1, EWI-F2, CNFD, SIFD and EXFD
are similar to the previous subsection and they are omitted here for brevity.

\begin{table}[t!]\small
\centering
\tabcolsep 0pt
\caption{Error analysis of MTI-FA for general nonlinearity:
$\fe^{\eps,\tau}(T)$ with $T=1$. }\label{MTI FA gen}
\begin{center}
\begin{tabular*}{1\textwidth}{@{\extracolsep{\fill}}llllllll}
\hline
$\fe^{\eps,\tau}(T)$& $\tau_0=0.2$	&$\tau_0/2^2$	
&$\tau_0/2^4$	&$\tau_0/2^6$	 &$\tau_0/2^8$ &$\tau_0/2^{10}$ &$\tau_0/2^{12}$\\
\hline
$\eps_0=1$	&1.97E\,--2	&1.22E\,--3	&7.35E\,--5	&4.54E\,--6	&2.83E\,--7	&1.78E\,--8	 &1.25E\,--9\\
	            %&           &2.00	    &2.03	    &2.01	    &2.00	    &2.00	    &1.92\\
$\eps_0/2$      &6.92E\,--3	&1.34E\,--3	&7.42E\,--5	&4.43E\,--6	&2.73E\,--7	&1.71E\,--8	 &1.19E\,--9\\
	            %&           &1.18	    &2.09	    &2.03	    &2.01	    &2.00	    &1.93\\
$\eps_0/2^2$	&1.61E\,--4	&4.01E\,--4	&4.04E\,--4	&2.63E\,--5	&1.66E\,--6	&1.04E\,--7	 &6.53E\,--9\\
	            %&           &\,--.66	    &\,--.01	    &1.97	    &1.99	    &2.00	    &2.00\\
$\eps_0/2^3$    &1.21E\,--2	&2.25E\,--3	&5.63E\,--4	&8.47E\,--5	&4.91E\,--6	&3.00E\,--7	 &1.84E\,--8\\
	            %&           &1.21	    &1.00	    &1.36	    &2.05	    &2.02	    &2.01\\
$\eps_0/2^4$	&9.04E\,--3	&9.78E\,--4	&1.68E\,--3	&1.50E\,--3	&1.58E\,--6	&5.97E\,--9	 &2.37E\,--9\\
	            %&           &1.60	    &\,--.39	    &0.08	    &4.94	    &4.02	    &0.67\\
$\eps_0/2^5$	&9.27E\,--3	&2.50E\,--4	&6.14E\,--6	&1.62E\,--3	&5.86E\,--5	&7.52E\,--6	 &4.87E\,--7\\
	            %&           &2.61	    &2.67	    &-4.02	    &2.39	    &1.48	    &1.97\\
$\eps_0/2^6$	&3.96E\,--3	&3.29E\,--4	&8.48E\,--6	&6.34E\,--7	&9.40E\,--4	&1.19E\,--4	 &1.91E\,--6\\
	            %&           &1.80	    &2.63	    &1.87	    &-5.26	    &1.49	    &2.98\\
$\eps_0/2^8$	&1.89E\,--3	&2.35E\,--4	&2.90E\,--5	&1.41E\,--7	&8.47E\,--7	&3.70E\,--7	 &5.17E\,--5\\
	            %&           &1.50	    &1.51	    &3.84	    &-1.29	    &0.60	    &-3.56\\
$\eps_0/2^{10}$ &1.27E\,--2	&8.46E\,--4	&5.46E\,--5	&6.29E\,--6	&1.26E\,--6	&1.27E\,--6	 &1.08E\,--6\\
	            %&           &1.95	    &1.98	    &1.56	    &1.16	    &\,--.01	    &0.12\\
$\eps_0/2^{12}$ &1.59E\,--4	&1.47E\,--4	&1.13E\,--5	&7.51E\,--7	&3.46E\,--8	&9.93E\,--8	 &3.49E\,--8\\
	            %&           &0.06	    &1.85	    &1.96	    &2.21	    &\,--.67	    &0.75\\
$\eps_0/2^{14}$ &9.89E\,--3	&5.33E\,--4	&3.18E\,--5	&1.96E\,--6	&1.17E\,--7	&1.72E\,--9	 &4.97E\,--9\\
	            %&           &0.06	    &1.85	    &1.96	    &2.21	    &\,--.67	    &0.75\\
\hline
\end{tabular*}
\end{center}
\end{table}

\begin{table}[t!]
\centering
\tabcolsep 0pt
\caption{Error analysis of MTI-F for general nonlinearity:
$\fe^{\eps,\tau}(T)$ with $T=1$.}\label{MTI F gen}
\begin{center}
\begin{tabular*}{1\textwidth}{@{\extracolsep{\fill}}llllllll}
\hline
$\fe^{\eps,\tau}(T)$      & $\tau_0=0.2$	&$\tau_0/2^2$	
&$\tau_0/2^4$	&$\tau_0/2^6$	 &$\tau_0/2^8$  &$\tau_0/2^{10}$   & $\tau_0/2^{12}$\\
\hline
$\eps_0=1$	      &5.79E\,--3	        &8.19E\,--4	
&5.28E\,--5	    &3.31E\,--6	    &2.07E\,--7	   &1.31E\,--8	      &9.53E-10\\	
$\eps_0/2$            &7.54E\,--3	        &1.28E\,--3	
&6.87E\,--5	    &3.93E\,--6	    &2.39E\,--7	   &1.50E\,--8	      &1.05E\,--9\\	
$\eps_0/2^2$	      &3.05E\,--2	        &3.58E\,--4	
&3.99E\,--4	    &2.61E\,--5	    &1.65E\,--6	   &1.03E\,--7	      &6.48E\,--9\\
$\eps_0/2^3$          &1.19E\,--2	        &2.81E\,--3	
&4.99E\,--4	    &8.07E\,--5	    &4.67E\,--6	   &2.85E\,--7	      &1.75E\,--8\\
$\eps_0/2^4$	      &8.83E\,--3	        &6.63E\,--4	
&1.43E\,--3	    &1.49E\,--3	    &1.28E\,--6	   &2.40E\,--8	      &3.48E\,--9\\
$\eps_0/2^5$	      &9.52E\,--3      	&3.02E\,--4	
&8.66E\,--5	    &1.54E\,--3	    &5.89E\,--5	   &7.52E\,--6	      &4.87E\,--7\\
$\eps_0/2^6$	      &3.76E\,--3      	&3.55E\,--4	
&4.82E\,--6	    &4.65E\,--6	    &9.35E\,--4	   &1.19E\,--4	      &1.91E\,--6\\
$\eps_0/2^8$	      &1.89E\,--3	        &2.41E\,--4	
&2.87E\,--5	    &2.55E\,--7	    &8.33E\,--7	   &3.91E\,--7	      &5.17E\,--5\\
$\eps_0/2^{10}$       &1.27E\,--2     	&8.46E\,--4	
&5.47E\,--5	    &6.33E\,--6	    &1.25E\,--6	   &1.27E\,--6	      &1.08E\,--6\\
$\eps_0/2^{12}$       &1.59E\,--4	        &1.47E\,--4	
&1.13E\,--5	    &7.51E\,--7	    &3.51E\,--8	   &9.88E\,--8	      &3.53E\,--8\\
$\eps_0/2^{14}$       &9.89E\,--3	        &5.33E\,--4	
&3.17E\,--5	    &1.95E\,--6	    &1.06E\,--7	   &9.43E\,--9	      &1.62E\,--8\\
                      \hline
\end{tabular*}
\end{center}
\end{table}\normalsize

From Tables \ref{MTI FA gen}\&\ref{MTI F gen} and additional results not
shown here for brevity, again we can see that both MTI-FA and MTI-F are
uniformly accurate for $0<\eps\le 1$, especially when $0<\eps\ll 1$.
In addition, the results suggest the following two independent error bounds
for both MTI-FA and MTI-F under a general nonlinearity in (\ref{WODE})
 \[|e^n|+\eps^2|\dot{e}^n|\lesssim\frac{\tau^2}{\eps^2},\qquad
  |e^n|+\eps^2|\dot{e}^n|\lesssim \tau^2+\eps^2,\qquad 0<\tau\leq\tau_0.\]
Rigorous justification for the above observation is still on-going.

\section{Conclusions}
\label{sec:7}

Different numerical methods were either designed or reviewed as well as compared with each other
for solving highly oscillatory second order differential equations
with a dimensionless parameter $0<\eps\le1$ which is inversely proportional to the speed of light,
especially in the nonrelativistic limit regime $0<\eps\ll1$. In this regime,
the solution propagates waves at wavelength $O(\eps^2)$ and amplitude at $O(1)$, which brings significantly numerical
burdens in practical computation.  Based on two types of multiscale
decomposition by either frequency or frequency and amplitude,
two multiscale time integrators (MTIs), e.g. MTI-FA and MTI-F,
were designed for solving the problem
when the nonlinearity is taken as either a pure power nonlinearity or a general
gauge invariant nonlinearity.  Two independent error bounds at $O(\tau^2/\eps^2)$
and $O(\eps^2)$ for $\eps\in(0,1]$ of the two MTIs were rigorously established
when the nonlinearity is taken as a pure power nonlinearity, which
immediately imply that the two MTIs converge uniformly with linear convergence rate
at $O(\tau)$ for  $\eps\in(0,1]$ and optimally with quadratic convergence rate at $O(\tau^2)$
in the regimes when either $\eps=O(1)$ or $0<\eps\le \tau$.
For comparison, classical methods, such as exponential wave integrators (EWIs)
and finite difference (FD) methods, were also presented for solving the problem.
Error bounds for them were given with explicitly dependence on the parameter $\eps$.
Based on these rigorous error estimates, we conclude that, in the regime $0<\eps\ll1$,
the $\eps$-scalability for the two MTIs is $\tau=O(1)$ which is independent of $\eps$,
where it is  at $\tau=O(\eps^2)$ and $\tau=O(\eps^3)$ for EWIs and FD methods, respectively.
Therefore, the proposed MTIs offer compelling advantages over those
classical methods in the regime $0<\eps\ll 1$.  Numerical results
confirmed our analytical error bounds. We remark here that both
MTI-FA and MTI-F and their error estimates can be extended to (\ref{WODE})
when $g(\rho)$ in (\ref{power}) is a polynomial in $\rho$. In the future,
we will extend the two MTIs to solve oscillatory ODEs from
molecular dynamics with high frequency \cite{Grimm,Lubich1,Lubich2,Sanz}
and nonlinear oscillatory dispersive partial differential equations arising from
plasma physics and general relativity \cite{Dong,Masmoudi}.

\appendix

\section{Detailed explicit formulas for the coefficients used in the MTIs}\label{ap: coff}

For the coefficients in (\ref{quad Ik J}) used in MTI-FA, after a detailed computation, we have
\begin{align*}
p_k=&\frac{\eps^2\omega\cos(\omega\tau)+i(2k+1)\sin(\omega\tau)-
\eps^2\omega\fe^{i(2k+1)\tau/\eps^2}}{(2k+1)^2\omega-\eps^4\omega^3},\\
\dot{p}_k=&\frac{i(2k+1)\cos(\omega\tau)-\eps^2\omega\sin(\omega\tau)-
i(2k+1)\fe^{i(2k+1)\tau/\eps^2}}{(2k+1)^2-\eps^4\omega^2},\qquad 1\le k\le p,\\
q_k=&\frac{\eps^2}{\omega\left[\eps^4\omega^2-(2k+1)^2\right]^2}
\biggl[i(4k+2)\eps^2\omega\cos(\omega\tau)-\left(\eps^4\omega^2+(2k+1)^2\right)
\sin(\omega\tau)\\
&+\left(\eps^4\omega^3\tau-(2k+1)^2\omega\tau-i(4k+2)
\eps^2\omega\right)\fe^{i(2k+1)\tau/\eps^2}\biggr],\\
\dot{q}_k=&\frac{1}{\left[\eps^4\omega^2-(2k+1)^2\right]^2}
\biggl[-\left(\eps^6\omega^2+(2k+1)^2\eps^2\right)\cos(\omega\tau)-i(4k+2)\eps^4\omega
\sin(\omega\tau)^{^{}}\\
&+\left(i(2k+1)\tau\eps^4\omega^2-i(2k+1)^3\tau+\eps^6\omega^2+(2k+1)^2\eps^2
\right)\fe^{i(2k+1)\tau/\eps^2}\biggr].
\end{align*}
Similarly, for the coefficients in (\ref{gen F appr}) used in MTI-F, we have
\begin{align*}
&c(\tau)=\frac{\nd_-\fe^{i\tau\nd_+}-\nd_+\fe^{i\tau\nd_-}+\nd_+-
\nd_-}{\eps^2(\nd_--\nd_+)\nd_+\nd_-},\qquad
\dot{c}(\tau)=i\frac{\fe^{i\tau\nd_+}-\fe^{i\tau\nd_-}}{\eps^2(\nd_--\nd_+)},\\
&d(\tau)=i\frac{\nd_-^2\fe^{i\tau\nd_+}-\nd_+^2\fe^{i\tau\nd_-}+i\tau\nd_+\nd_-(\nd_+
-\nd_-)+\nd_+^2-\nd_-^2}{\eps^2(\nd_+-\nd_-)\nd_+^2\nd_-^2},\qquad
\dot{d}(\tau)=c(\tau).
\end{align*}

\section{Proof of Theorem \ref{tm1}}
\label{Ap proof MTI_FA}

In order to proceed with the proof, we introduce the following auxiliary problem
\begin{equation}\label{B1}
\left\{
  \begin{split}
    & \eps^2\ddot{\tilde{y}}^n(s)+\left(\alpha+\frac{1}{\eps^2}\right)\tilde{y}^n(s)
    +g\left(|\tilde{y}^n(s)|^2\right)\tilde{y}^n(s)=0,\qquad s>0,\\
    & \tilde{y}^n(0)= y^n,\qquad\dot{\tilde{y}}^n(0)=\dot{y}^n,\qquad n=0,1,\ldots,
  \end{split}
\right.
\end{equation}
and denote two local errors and an error energy as
\begin{align}\label{etans12}
& \eta^{n}(s):=y(t_{n}+s)-\tilde{y}^n(s),\qquad \dot{\eta}^n(s):=
\dot{y}(t_{n}+s)-\dot{\tilde{y}}^n(s),\qquad s\geq 0,\\
& \xi^{n+1}:=\tilde{y}^n(\tau)-y^{n+1},\qquad \dot{\xi}^{n+1}:=
\dot{\tilde{y}}^n(\tau)-\dot{y}^{n+1},\label{xins12}\\
\label{error_energy}& \mathcal{E}\left(e,\dot{e}\right):=\eps^2\left|\dot{e}\right|^2
+\left(\alpha+\frac{1}{\eps^2}\right)\left|e\right|^2,\qquad \forall\,e,\, \dot{e}\in\bC.
\end{align}
Noticing (\ref{error_fun}) and using the triangle inequality, we have
\begin{align}
  & \eta^n(0)=e^n,\qquad \dot{\eta}^n(0)=\dot{e}^n,\\
  \label{3xin}&\left|e^{n+1}\right|\leq \left|\eta^n(\tau)\right|+\left|\xi^{n+1}\right|,\qquad \left|\dot{e}^{n+1}\right|\leq\left|\dot{\eta}^n(\tau)\right|+\left|\dot{\xi}^{n+1}\right|.
\end{align}

Before we present the detailed proof, we first establish the following lemmas.

\begin{lemma}\label{lm1}
For any $n=0,1,\ldots,$ we have
\begin{equation}
  \mcE\left(e^{n+1},\dot{e}^{n+1}\right)\leq (1+\tau)\mcE\left(\eta^{n}(\tau),\dot{\eta}^{n}(\tau)\right)
  +\left(1+\frac{1}{\tau}\right)\mcE\left(\xi^{n+1},\dot{\xi}^{n+1}\right).
\end{equation}
\end{lemma}
\begin{proof}
  Noticing (\ref{error_energy}), (\ref{3xin}), the above inequality follows
  by using the Young inequality.\qed
\end{proof}

Let $C_0$ be given in (\ref{C0}) and define
\begin{equation}\label{tau0}
\tau_1:=\left(2C_0+4\right)^{-1}K_1^{-1},\qquad \mbox{with}\qquad K_1=\left\|g(\cdot)\right\|_{L^\infty\left(0,\left(2C_0+4\right)^2\right)}. \end{equation}

\begin{lemma}\label{lm: reg y}
For the problem (\ref{B1}), if (\ref{FSWDbound}) holds for any fixed $n$
($n=0,1,\ldots,\frac{T}{\tau}-1$),
which will be proved by an induction argument later, then we have
\begin{equation} \label{bndytn}
\left\|\tilde{y}^n\right\|_{L^\infty(0,\tau)}\leq 2C_0+3\lesssim1,
\qquad  0\le \tau\leq\tau_1.
\end{equation}
\end{lemma}

\begin{proof}
By using the variation-of-constant formula to (\ref{B1}), we get for $0\le s\le \tau_1$
\begin{equation} \label{yns843}
\displaystyle
\tilde{y}^n(s)=\cos(\omega s)y^n+\frac{\sin(\omega s)}{\omega}\dot{y}^n-
 \int_0^s\frac{\sin(\omega(s-\theta))}{\eps^2\omega}g(|\tilde{y}^n(\theta)|^2)\tilde{y}^n(\theta)d\theta.
\end{equation}
Then the rest of the proof proceeds in the analogous lines
as in \cite{Tao} for nonlinear dispersive and wave equations by using
 the bootstrap principle and noticing (\ref{FSWDbound}). \qed
\end{proof}

\begin{lemma}\label{lm:stab_AP}
Under the same assumption as in Lemma \ref{lm: reg y}, we have
for $n=0,1,\ldots,\frac{T}{\tau}-1$,
\begin{equation}\label{A_pre1}
  \mcE\left(\eta^{n}(\tau),\dot{\eta}^{n}(\tau)\right)
  -\mcE\left(e^{n},\dot{e}^{n}\right)\lesssim \tau \mcE\left(e^{n},\dot{e}^{n}\right), \qquad
  0\leq\tau\leq \tau_1.
\end{equation}
\end{lemma}

\begin{proof}
\label{Ap2}
Subtracting  (\ref{B1}) from (\ref{WODE}) and noticing (\ref{etans12}), we obtain
\begin{equation}
\left\{
  \begin{split}
    & \eps^2\ddot{\eta}^n(s)+\left(\alpha+\frac{1}{\eps^2}\right)\eta^n(s)
    +\tilde{g}^n(s)=0,\qquad s>0,\\[0.5em]
    & \eta^n(0)= e^n,\qquad\dot{\eta}^n(0)=\dot{e}^n,\qquad n=0,1,\ldots;
  \end{split}
\right.
\end{equation}
where
\begin{equation}
\tilde{g}^n(s)= g\left(|y(t_n+s)|^2\right)y(t_n+s)-g\left(|\tilde{y}^n(s)|^2\right)\tilde{y}^n(s).
\end{equation}
By using the variation-of-constant formula and the triangle inequality, we get
\begin{equation}\label{etantau31}
\left\{
\begin{split}
 &\left|\eta^n(\tau)\right|\leq\left|\cos(\omega\tau)e^n+\frac{\sin(\omega\tau)}{\omega}\dot{e}^n\right|+
\int_0^\tau \left|\frac{\sin(\omega(\tau-s))}{\eps^2\omega}\tilde{g}^n(s) \right| ds,\\
  &\left|\dot{\eta}^n(\tau)\right|\leq\left|-\omega\sin(\omega\tau)e^n+\cos(\omega\tau)\dot{e}^n\right|+
 \int_0^\tau \left|\frac{\cos(\omega(\tau-s))}{\eps^2}\tilde{g}^n(s) \right| ds.
 \end{split}\right.
\end{equation}
From (\ref{error_energy}) with $e=e^n$ and $\dot{e}=\dot{e}^n$, we have
\begin{equation*}
\mcE\left(e^n,\dot{e}^n\right)=\eps^2\left|-\omega\sin(\omega\tau)e^n+\cos(\omega\tau)\dot{e}^n\right|^2
+\left(\alpha+\frac{1}{\eps^2}\right)\left|\cos(\omega\tau)e^n+\frac{\sin(\omega\tau)}{\omega}\dot{e}^n\right|^2.
\end{equation*}
From (\ref{etantau31}) and (\ref{error_energy}), noticing the above equality
and using  the Young inequality, we get
\begin{equation}\label{plm3-1}
\mcE\left(\eta^n(\tau),\dot{\eta}^n(\tau)\right)-(1+\tau)\mcE\left(e^n,\dot{e}^n\right)
\leq \left(1+\frac{1}{\tau}\right)\left(\alpha+\frac{2}{\eps^2}\right)\left(\int_0^\tau \left|\tilde{g}^n(s)\right|ds\right)^2.
\end{equation}
Noticing (\ref{C0}) and (\ref{bndytn}), we have
\begin{equation}\label{gntd124}
  \left|\tilde{g}^n(s)\right|\lesssim \left|\eta^n(s)\right|,\qquad 0\le s\le \tau\le \tau_1.
\end{equation}
Plugging (\ref{gntd124}) into (\ref{plm3-1}), noticing (\ref{error_energy}) and
using the H\"{o}lder inequality, we get
\begin{eqnarray}
\mcE\left(\eta^n(\tau),\dot{\eta}^n(\tau)\right)-(1+\tau)\mcE\left(e^n,\dot{e}^n
\right)&\leq&\left(1+\frac{1}{\tau}\right)\left(\alpha+\frac{2}{\eps^2}\right)
\tau\int_0^\tau \left|\eta^n(s)\right|^2ds\nonumber\\
&\lesssim&\int_0^\tau \mcE\left(\eta^n(s),\dot{\eta}^n(s)\right)ds,
\qquad 0\le \tau\le \tau_1. \qquad
\end{eqnarray}
Then the estimate (\ref{A_pre1}) can be obtained by applying the Gronwall inequality.
\qed
\end{proof}

\begin{lemma}(A prior estimate of MDFA)\label{lm:reg_F}
Let ${z}_\pm^n(s)$ and $r^n(s)$ be the solution of the MDFA
(\ref{pLSADz31}) under the initial conditions (\ref{FSW-i51})
with $z_\pm^n(0)=z_\pm^{(0)}$ and $\dot{r}^n(0)=\dot{r}^{(0)}$
defined in (\ref{F_ini1}). Under the same assumption
as in Lemma \ref{lm: reg y}, there exists a constant $\tau_2>0$,
independent of $\eps$ and  $n$, such that for
$0<\tau\leq\tau_2$ and all $n=0,1,\ldots,\frac{T}{\tau}-1$,
  \begin{equation}
    \label{lm2_1}\left\|\frac{d^m}{dt^m} {z}_\pm^n\right\|_{L^\infty(0,\tau)}
    \lesssim 1,\ m=0,1,2,3;\qquad
    \eps^{2l-2}\left\|\frac{d^l}{dt^l} {r}^n\right\|_{L^\infty(0,\tau)}
    \lesssim1,\ l=0,1,2.
  \end{equation}
\end{lemma}

\begin{proof}
From (\ref{FSW-i51}), noticing (\ref{FSWDbound}),
 (\ref{F_ini1}) and (\ref{z_pm exact}), we obtain
\begin{equation}\label{lm2dz}
\left|z_{\pm}^{(0)}\right|\lesssim1, \qquad\qquad\left|\dot{z}_{\pm}^{(0)}\right|\lesssim1,\qquad \qquad
\left\|\frac{d^m}{dt^m} {z}_\pm^n\right\|_{L^\infty(0,\infty)}\lesssim1,\qquad m=0,1,2,3.
\end{equation}
To estimate $r^n(s)$ in (\ref{pLSADz31}), using the variation-of-constant formula,
noting (\ref{F_ini1}), (\ref{gr def}) and (\ref{gkpm321}), we get
\begin{equation}\label{VCFr s}
r^n(s)=\frac{\sin(\omega s)}{\omega}\dot{r}^{(0)}-
\sum_{k=1}^p\left[I^n_{k,+}(s)+\overline{I^n_{k,-}(s)}\right]-J^n(s), \qquad s\ge0,
\end{equation}
where
\begin{equation} \label{Ikpmns2}
  \left\{
  \begin{split}
    &I_{k,\pm}^n(s):= \int_0^s\frac{\sin(\omega(s-\theta))}
  {\eps^2\omega}\fe^{i(2k+1)\theta/\eps^2}g_{k,\pm}^n(\theta) d\theta,\\
    &J^n(s):= \int_0^s\frac{\sin(\omega(s-\theta))}{\eps^2\omega}
    \left[h^n(\theta)+\eps^2u^n(\theta)\right]d\theta, \qquad s\ge0.
  \end{split}
  \right.
\end{equation}
Plugging (\ref{lm2dz}) into (\ref{F_ini1}) and using the triangle inequality, we have
\begin{equation}\label{r0876}
\left|\dot{r}^{(0)}\right|\leq\left|\dot{{z}}_+^{(0)}\right|+\left|\dot{ {z}}_-^{(0)}\right|\lesssim1.
\end{equation}
Let
\begin{equation}\label{Tktht4}
 T_k(\theta)=\frac{\eps^2\fe^{i(2k+1)\theta/\eps^2}}{\eps^4\omega^2-(2k+1)^2}
\left[\cos\left(\omega(s-\theta)\right)+\frac{i(2k+1)}{\eps^2\omega}
\sin\left(\omega(s-\theta)\right)\right]=O(\eps^2),
\end{equation}
then we have
\begin{equation}\label{T_k}
\frac{d}{d\theta}T_k(\theta)=\frac{\sin(\omega(s-\theta))}
{\eps^2\omega}\fe^{i(2k+1)\theta/\eps^2}=O(1),\qquad k=1,2,\ldots,p.
\end{equation}
Plugging (\ref{T_k}) into (\ref{Ikpmns2}), noticing (\ref{Tktht4}),
(\ref{gkpm321}), (\ref{g_k def}) and (\ref{lm2dz}), we get
\begin{eqnarray}\label{Ik}
\left|I^n_{k,\pm}(s)\right|&=&\left|\int_0^sg_{k,\pm}^n
(\theta)\frac{d}{d\theta}T_k(\theta)d\theta\right|
=\left|\left.g_{k,\pm}^n(\theta)T_k(\theta)\right|_0^s-
\int_0^sT_k(\theta)\frac{d}{d\theta}g_{k,\pm}^n(\theta)\,d\theta\right|\nonumber\\
&\lesssim&\eps^2+\int_0^s\eps^2\,ds=\eps^2(1+s), \qquad s\ge0.
\end{eqnarray}
From (\ref{Ikpmns2}), noting (\ref{fr_tilde}),
(\ref{gkpm321}), (\ref{T_k}) and (\ref{lm2dz}), we obtain for $s\ge0$
\begin{equation} \label{Jns654}
\left|J^n(s)\right|\lesssim \int_0^s\left[\eps^2|u^n(\theta)|+|h^n(\theta)|\right]d\theta
\lesssim \eps^2 s+\int_0^s\left|h\left( {z}_+^n(\theta),{z}_-^n(\theta),
{r}^n(\theta);\theta\right)\right|d\theta.
\end{equation}
Plugging (\ref{Ik}), (\ref{r0876}) and (\ref{Jns654}) into (\ref{VCFr s}), we have
\begin{equation} \label{rns987}
|{r}^n(s)|\lesssim \eps^2(1+s)+\int_0^s\left|h\left( {z}_+^n(\theta),{z}_-^n(\theta),
{r}^n(\theta);\theta\right) \right|d\theta, \qquad s\ge0.
\end{equation}
By using the bootstrap argument to (\ref{rns987}) \cite{Tao}, noting (\ref{lm2dz}) and (\ref{h def}),
there exists a constant $\tau_2>0$ independent of $\eps$ and  $n$, such that for
$0<\tau\leq\tau_2$ and all $n=0,1,\ldots,\frac{T}{\tau}-1$,
\begin{equation}\label{rns864}
\left\| {r}^n\right\|_{L^\infty(0,\tau)}
\lesssim\eps^2, \qquad \left\|\dot{ {r}}^n\right\|_{L^\infty(0,\tau)}
\lesssim1,\qquad \left\|\ddot{ {r}}^n\right\|_{L^\infty(0,\tau)}
\lesssim\eps^{-2}.
\end{equation}
The proof is completed by combining (\ref{lm2dz}) and (\ref{rns864}). \qed
\end{proof}

\begin{lemma}(Estimate on local error $\xi^{n+1}$)\label{lm:local_error_F}
Under the same assumption as in Lemma \ref{lm: reg y},
for any $n=0,1,\ldots,\frac{T}{\tau}-1,$, we have the following two independent bounds
\begin{equation}\label{A_pre2}
  \mcE\left(\xi^{n+1},\dot{\xi}^{n+1}\right)\lesssim \frac{\tau^6}{\eps^6},\qquad
  \mcE\left(\xi^{n+1},\dot{\xi}^{n+1}\right)\lesssim \tau^2\eps^2, \qquad 0\leq\tau\leq \tau_2.
\end{equation}
\end{lemma}

\begin{proof}
Similar to sections 2\&3, we can solve the problem (\ref{B1}) analytically via MDFA
and obtain
\begin{equation}\label{tdyns345}
\tilde{y}^n(\tau)=e^{i\tau/\eps^2}z_+^n(\tau)+e^{-i\tau/\eps^2}\overline{z_-^n}(\tau)
+r^n(\tau),
\end{equation}
where $z_\pm^n(\tau)$ and $r^n(\tau)$ are defined as (\ref{zpmnt876}) and (\ref{VCFr}), respectively
with $\phi_1^n=y^n$ and $\phi_2^n=\eps^2 \dot{y}^n$ in (\ref{FSW-i51}).
Plugging (\ref{tdyns345}) and (\ref{IFSW}) into (\ref{xins12}), noting (\ref{FS_r}), we have
\begin{eqnarray}\label{xi}
\xi^{n+1}&=&\fe^{i\tau/\eps^2}\left( {z}_+^n(\tau)-z_+^{n+1}\right)+
\fe^{-i\tau/\eps^2}\left(\overline{ {z}_-^n(\tau)}-\overline{z_-^{n+1}}\right)
+ {r}^n(\tau)-r^{n+1}\nonumber\\
&=&{r}^n(\tau)-r^{n+1}=\mathcal{J}^n+
\sum_{k=1}^p\left[\mathcal{I}^n_{k,+}+\overline{\mathcal{I}^n_{k,-}}\right],
\end{eqnarray}
where
\begin{equation} \label{mJn367}
\mathcal{J}^n:=\frac{\tau\sin(\omega\tau)}{2\omega}u^{(0)}-J^n,
\qquad \mathcal{I}^n_{k,\pm}:= p_k g^{(0)}_{k,\pm}+q_k\dot{g}^{(0)}_{k,\pm}-I^n_{k,\pm},\qquad k=1,\ldots,p.
\end{equation}
From (\ref{mJn367}), noting (\ref{Ikpmn965}) and (\ref{quad Ik J})
where the Gautschi type or trapezoidal quadrature
was used to approximate integrals and using the Taylor expansion,
we obtain for $0<\tau\le \tau_2$
\begin{equation}\label{C6}
\left|\mathcal{I}_{k,\pm}^n\right| =\left|\frac{1}{2}\int_0^\tau \theta^2 \frac{\sin(\omega(\tau-\theta))}{\eps^2\omega}\fe^{i(2k+1)\theta/\eps^2}
\ddot{g}_{k,\pm}^n(t(\theta)) d\theta\right|\lesssim \int_0^\tau \theta^2d\theta
\lesssim \tau^3,
\end{equation}
where $0\leq t(\theta)\leq\tau$. In addition, similar to (\ref{Ik})
by using integration by parts, we have
\begin{equation}\label{Ink}
\left|\mathcal{I}_{k,\pm}^n\right|=\left|\frac{1}{2}\int_0^\tau \theta^2\ddot{g}_{k,\pm}(t(\theta))\frac{d}{d\theta}T_k(\theta)
d\theta\right|\lesssim \tau^2\eps^2, \qquad 0<\tau\le \tau_2.
\end{equation}
Similarly, we can get two independent bounds for $\mathcal{J}^n$ as
\begin{equation}\label{Pn-2}
\left|\mathcal{J}^n \right|\lesssim \frac{\tau^3}{\eps^2},
\qquad \qquad \left|\mathcal{J}^n\right|\lesssim\tau\eps^2, \qquad 0<\tau\le \tau_2.
\end{equation}
From (\ref{xi}), noting (\ref{C6}), (\ref{Ink}) and (\ref{Pn-2}), we get two independent bounds
for $\xi^{n+1}$ as
\begin{equation} \label{xin786}
  \left|\xi^{n+1}\right|
  \lesssim \tau^3+\frac{\tau^3}{\eps^2}\lesssim \frac{\tau^3}{\eps^2},
  \qquad  \left|\xi^{n+1}\right|
  \lesssim \eps^2\tau +\tau^2\eps^2 \lesssim \tau \eps^2, \qquad 0<\tau\le \tau_2.
\end{equation}
Similar to the above, we can obtain two independent bounds for $\dot{\xi}^{n+1}$ as
\begin{equation}\label{dxin}
  \left|\dot{\xi}^{n+1}\right|\lesssim \frac{\tau^3}{\eps^4},\qquad\qquad
 \left|\dot{\xi}^{n+1}\right|\lesssim \tau, \qquad 0<\tau\le \tau_2.
\end{equation}
Then (\ref{A_pre2}) is a combination of (\ref{xin786})
and (\ref{dxin}) by noting (\ref{error_energy}). \qed
\end{proof}

Combining Lemmas \ref{lm1}, \ref{lm: reg y}, \ref{lm:stab_AP} and \ref{lm:local_error_F},
we can prove the Theorem \ref{tm1}.

\smallskip

\noindent {\it Proof of Theorem \ref{tm1}}
The proof proceeds by using the energy method with the help of the method of
mathematical induction for establishing uniform boundedness of
$y^n$ and $\dot{y}^n$ \cite{Bao5,Cai1,Dong}.

Since $e^0=0$ and $\dot{e}^0=0$, $y^0=y(0)$ and $\dot{y}^0=\dot{y}(0)$,
noting (\ref{C0}), we can get
that (\ref{FSWDerror1})-(\ref{FSWDbound}) hold for $n=0$.

Now assuming that (\ref{FSWDerror1})-(\ref{FSWDbound})
are valid for all $0\leq n\leq m-1\leq \frac{T}{\tau}-1$,
we need to show they are still valid for $n=m$.
From Lemmas \ref{lm1} and \ref{lm:stab_AP}, we have
\begin{eqnarray}
  \mcE\left(e^{n+1},\dot{e}^{n+1}\right)-\mcE\left(e^n,\dot{e}^n\right)
  \lesssim \tau \mcE \left(e^n,\dot{e}^n\right) + \frac{1}{\tau}\mcE
  \left(\xi^{n+1},\dot{\xi}^{n+1}\right),\qquad 0<\tau\leq \tau_1.
\end{eqnarray}
Summing the above inequality for $n=0,1,\ldots,m-1$,
noticing $\mcE\left(e^{0},\dot{e}^{0}\right)=0$, we obtain
\begin{equation}\label{emem367}
\mcE\left(e^{m},\dot{e}^{m}\right)\lesssim \tau
\sum_{l=1}^{m-1}\mcE\left(e^{l},\dot{e}^{l}\right)+
\frac{1}{\tau}\sum_{l=1}^m \mcE\left(\xi^l,\dot{\xi}^l\right).
\end{equation}
Applying the discrete Gronwall inequality to (\ref{emem367}), we get
\begin{equation}\label{emem478}
\mcE \left(e^{m},\dot{e}^{m}\right)\lesssim \frac{1}{\tau}
\sum_{l=1}^m \mcE\left(\xi^l,\dot{\xi}^l\right).
\end{equation}
Plugging (\ref{A_pre2}) into (\ref{emem478}), we get two independent bounds as
\begin{equation}\label{emem941}
\mcE \left(e^{m},\dot{e}^{m}\right)\lesssim
\frac{T}{\tau^2}\frac{\tau^6}{\eps^6}\lesssim \frac{\tau^4}{\eps^6} ,\quad
\mcE \left(e^{m},\dot{e}^{m}\right)\lesssim \frac{T}{\tau^2} \tau^2\eps^2\lesssim
\eps^2,\quad 0<\tau\leq \min\{\tau_1,\tau_2\}.
\end{equation}
Combing (\ref{emem941}) and (\ref{error_energy}), we get
\[|e^m|\le \eps \sqrt{\mcE \left(e^{m},\dot{e}^{m}\right)}\lesssim \frac{\tau^2}{\eps^2},
\quad \eps^2|\dot{e}^m|\le
\eps \sqrt{\mcE \left(e^{m},\dot{e}^{m}\right)}\lesssim \frac{\tau^2}{\eps^2},
\qquad |e^m|\lesssim \eps^2, \quad \eps^2|\dot{e}^m|\lesssim \eps^2,
\]
which immediately imply that (\ref{FSWDerror1}) is valid for $n=m$.
In addition, we have \cite{Deg,Jin}
\begin{equation}
\left|y^m\right|-C_0\leq \left|e^m\right| \lesssim \min_{0<\eps\le 1}
\left\{\frac{\tau^2}{\eps^2},\tau^2\right\}\lesssim \tau,\qquad
  \eps^2\left|\dot{y}^m\right|-C_0\leq \eps^2 \left|\dot{e}^m\right|\lesssim \tau.
\end{equation}
Thus there exists a $\tau_3>0$ independent of $\eps$ and $m$, such that
\[\left|y^m\right|\le C_0+1, \qquad |\dot{y}^m|\leq \frac{C_0+1}{\eps^2}.\]
Thus (\ref{FSWDbound}) is valid for $n=m$. By the method of mathematical induction,
the proof is completed if we choose
$\tau_0=\min\left\{\tau_1,\tau_2,\tau_3\right\}$. \qed

\section{Proof of Theorem \ref{tm2}}
\label{Ap proof MTI_F}
The proof is quite similar to that of Theorem \ref{tm1}.
Following the same notations introduced before, let $y^n$ and $\dot{y}^n$ in (\ref{B1})
be obtained by the method MTI-F and
assume (\ref{LSADbound}) holds, then
the regularity and stability results, i.e., Lemmas \ref{lm1}-\ref{lm:stab_AP},
for the auxiliary problem (\ref{B1}) still hold.

\begin{lemma}(A prior estimate of MDF)\label{lm:reg_L}
Let ${z}_\pm^n(s)$ and $r^n(s)$ be the solution of
the MDF (\ref{pLSADz1}) under the initial conditions (\ref{FSW-i21})
with $z_\pm^n(0)=z_\pm^{(0)}$, $\dot{z}_\pm^n(0)=\dot{z}_\pm^{(0)}$
and $\dot{r}^n(0)=\dot{r}^{(0)}$
defined in (\ref{F_ini1}).
Under the assumption (\ref{LSADbound}), there exists a
constant $\tau_4>0$ independent of $\eps$ and
 $n$, such that for $0<\tau\leq\tau_4$
and all $n=0,1,\ldots,\frac{T}{\tau}-1$
  \begin{eqnarray}\label{reg F}
   \left\|\frac{d^m}{dt^m} {z}_\pm^n\right\|_{L^\infty(0,\tau)}+
   \eps^2\left\|\frac{d^3}{dt^3} {z}_\pm^n\right\|_{L^\infty(0,\tau)}+
   \eps^{2m-2}\left\|\frac{d^m}{dt^m} {r}^n\right\|_{L^\infty(0,\tau)}\lesssim1,
   \quad m=0,1,2. \qquad
  \end{eqnarray}
\end{lemma}

\begin{proof}
For the estimates on $z^n_{\pm}(s)$, we refer the readers
to \cite[Appendix]{Cai1} and omit the details here for brevity.
For the estimates on $r^n(s)$, we can have a similar
variation-of-constant formula as (\ref{VCFr s}) but without the
term $u^n$ defined in $J^n(s)$.
Then the rest part of the proof can be done in
the same manner as Lemma \ref{lm:reg_F}.
\end{proof}

\begin{lemma}(Estimate on local error $\xi^{n+1}$)\label{lm:local_error_L}
Under the same assumption as in Lemma \ref{lm: reg y} and assume (\ref{LSADbound}) holds,
for any $n=0,1,\ldots,\frac{T}{\tau}-1$, we have two independent bounds
\begin{equation}
  \mcE\left(\xi^{n+1},\dot{\xi}^{n+1}\right)\lesssim
  \frac{\tau^6}{\eps^6},\qquad
  \mcE\left(\xi^{n+1},\dot{\xi}^{n+1}\right)\lesssim
  \frac{\tau^6}{\eps^2}+\tau^2\eps^2,\qquad 0<\tau\leq \tau_4.
  \label{lm:local_error_L result}
\end{equation}
\end{lemma}

\begin{proof}
Again, similar to sections 2\&3, we can solve the problem (\ref{B1}) analytically via MDF
and obtain that $\tilde{y}^n(\tau)$ satisfies (\ref{tdyns345})
with $z_\pm^n(\tau)$ and $r^n(\tau)$  defined as (\ref{2VCFz}) and (\ref{r gr}) with $u^n=0$, respectively
with $\phi_1^n=y^n$ and $\phi_2^n=\eps^2 \dot{y}^n$ in (\ref{FSW-i21}).
Plugging (\ref{tdyns345}) and (\ref{IFSW}) into (\ref{xins12}), using the triangle inequality, we get
\begin{eqnarray}\label{xi6789}
|\xi^{n+1}|&=&\left|\fe^{i\tau/\eps^2}\left( {z}_+^n(\tau)-z_+^{n+1}\right)+
\fe^{-i\tau/\eps^2}\left(\overline{ {z}_-^n(\tau)}-\overline{z_-^{n+1}}\right)
+ {r}^n(\tau)-r^{n+1}\right|\nonumber\\
&\le&\left| {z}_+^n(\tau)-z_+^{n+1}\right|+\left|{z}_-^n(\tau)
-z_-^{n+1}\right|+\left|{r}^n(\tau)-r^{n+1}\right|.
\end{eqnarray}
Similar to the proof in Lemma \ref{lm:local_error_F},
we obtain the following two independent bounds
\begin{eqnarray} \label{rnrnp1378}
\left|{r}^n(\tau)-r^{n+1}\right|\lesssim \frac{\tau^3}{\eps^2}, \qquad
\left|{r}^n(\tau)-r^{n+1}\right|\lesssim \tau \eps^2, \qquad 0<\tau\le\tau_4.
\end{eqnarray}
Subtracting $z_\pm^{n+1}$ in (\ref{F s}) from (\ref{2VCFz tau}), using
the Taylor expansion, and noting (\ref{asbslbd}), (\ref{fnpms613}) and (\ref{reg F}), we get
\begin{equation}\label{zpmn479}
 \left|{z}_\pm^n(\tau)-z_\pm^{n+1}\right|=\frac{1}{2}\left|\int_0^{\tau} \theta^2 b(\tau-\theta)\ddot{f}_\pm^n\left(t(\theta)\right)d\theta\right|\lesssim
 \int_0^\tau \theta^2\,d\theta\lesssim \tau^3,
\end{equation}
where $0\leq t(\theta)\leq\tau$. Inserting (\ref{zpmn479}) and (\ref{rnrnp1378}) into
(\ref{xi6789}), we obtain two independent bounds for $\xi^{n+1}$ as
\begin{equation}\label{xin1854}
 \left|\xi^{n+1}\right|\lesssim\frac{\tau^3}{\eps^2},\qquad
 \left|\xi^{n+1}\right|\lesssim\tau^3+\tau\eps^2, \qquad 0<\tau\leq \tau_4.
\end{equation}
Similarly, we can get two independent bounds for $\dot{\xi}^{n+1}$ as
\begin{equation}\label{dxin1876}
\left|\dot{\xi}^{n+1}\right|\lesssim\frac{\tau^3}{\eps^4},\qquad
\left|\dot{\xi}^{n+1}\right|\lesssim\frac{\tau^3+\tau\eps^2}{\eps^2},
\qquad 0<\tau\leq \tau_4.
\end{equation}
Then (\ref{lm:local_error_L result}) is a combination of (\ref{xin1854})
and (\ref{dxin1876}) by noting (\ref{error_energy}).\qed
\end{proof}

Combining Lemmas \ref{lm1}, \ref{lm: reg y},
\ref{lm:stab_AP} and \ref{lm:local_error_L}, we can prove the Theorem \ref{tm2}.

\smallskip

\noindent {\it Proof of Theorem \ref{tm2}} The argument proceeds in
analogous lines as for the Theorem \ref{tm1} and we omit the details here for brevity. \qed

%%%% Acknowledgments %%%%%%%%
\section*{Acknowledgments}
This work was supported by the Singapore A*STAR SERC  PSF-Grant	1321202067.
Part of this work was done when the authors were visiting the Institute for
Mathematical Science, National University of Singapore in 2011/12.

%%%% Bibliography  %%%%%%%%%%

\end{document}